\documentclass{amsart}
\usepackage{amsfonts}
\usepackage[margin=3cm]{geometry}
\usepackage[all,cmtip]{xy}
\usepackage{enumitem}
\usepackage{amsmath}
\usepackage{mathtools}
\usepackage{hyperref}
\usepackage{amssymb}
\usepackage{graphicx}
\usepackage{tikz-cd}
\usepackage{cleveref}

\theoremstyle{plain}
\newtheorem{theorem}{Theorem}
\newtheorem{theoremann}{Theorem}  

\newtheorem{conjecture}[theorem]{Conjecture}
\newtheorem{corollary}[theorem]{Corollary}

\newtheorem{lemma}[theorem]{Lemma}

\newtheorem{proposition}[theorem]{Proposition}

\theoremstyle{definition}
\newtheorem{definition}[theorem]{Definition}
\newtheorem{definitionp}[theorem]{Definition/Proposition}
\newtheorem{example}[theorem]{Example}
\newtheorem{remark}[theorem]{Remark}

\newtheorem{construction}[theorem]{Construction}
\newtheorem*{acknowledgement}{Acknowledgement}

\numberwithin{equation}{section}
\numberwithin{theorem}{section}

\makeatletter
\newcommand{\myitem}[1]{%
\item[#1]\protected@edef\@currentlabel{#1}%
}
\makeatother

\let\cal\mathcal
\def\AA{{\cal A}}
\def\BB{{\cal B}}
\def\CC{{\cal C}}
\def\DD{{\cal D}}
\def\EE{{\cal E}}
\def\FF{{\cal F}}
\def\OO{{\cal O}}
\def\VV{{\cal V}}

\DeclareMathOperator{\Acb}{\textbf{Ac}^b}
\DeclareMathOperator{\Kb}{\textbf{K}^b}
\DeclareMathOperator{\Db}{\mathbf{D}^b}
\DeclareMathOperator{\DAb}{\mathbf{D}_{\mathcal{A}}^b}
\DeclareMathOperator{\Dbinf}{\mathbf{D}^b_\infty}

\DeclareMathOperator{\Pl}{{\mathcal{P}\ell}}
\DeclareMathOperator{\Plfin}{{\mathcal{P}\ell}_{\operatorname{fin}}}

\DeclareMathOperator{\Hom}{Hom}
\DeclareMathOperator{\Mod}{Mod}
\DeclareMathOperator{\smod}{mod}
\DeclareMathOperator{\im}{im}

\DeclareMathOperator{\coker}{coker}

\newcommand{\inflation}{\rightarrowtail}
\newcommand{\deflation}{\twoheadrightarrow}

\newcommand{\ex}[1]{{#1}^\text{ex}}

\newcommand{\T}{\mathsf{T}}
\DeclareMathOperator{\DTb}{\mathbf{D}_{\T}^b}
\DeclareMathOperator{\tor}{tor}
\newcommand{\fA}{\mathfrak{A}}
\newcommand{\bA}{\mathbb{A}}
\newcommand{\bK}{\mathbb{K}}
\newcommand{\bQ}{\mathbb{Q}}

\newcommand{\bZ}{\mathbb{Z}}

\begin{document}
\title[Non-commutative id\`{e}le class group]{A non-commutative analogue of Clausen's view \linebreak on the id\`{e}le class group}
\author{Oliver Braunling}
\address{Oliver Braunling \\ Albert-Ludwigs-University Freiburg \\ Institute for Mathematics \\ D-79104 Freiburg \\ Germany}
\email{oliver.braeunling@math.uni-freiburg.de}
\author{Ruben Henrard}
\address{Ruben Henrard \\ Universiteit Hasselt \\ Campus Diepenbeek \\ Departement WNI \\ 3590 Diepenbeek \\ Belgium}
\email{ruben.henrard@uhasselt.be}
\author{Adam-Christiaan van Roosmalen}
\address{Adam-Christiaan van Roosmalen \\ Universiteit Hasselt \\ Campus Diepenbeek \\ Departement WNI \\ 3590 Diepenbeek \\ Belgium}
\email{adamchristiaan.vanroosmalen@uhasselt.be}
\thanks{The first author was supported by DFG GK1821 \textquotedblleft Cohomological Methods
in Geometry\textquotedblright.}
\thanks{The third author was supported by FWO (12.M33.16N)}


\begin{abstract}
Clausen predicted that Chevalley's id\`{e}le class group of a number field
$F$ appears as the first $K$-group of the category of locally compact
$F$-vector spaces. This has turned out to be true, and even
generalizes to the higher $K$-groups in a suitable sense. We replace $F$ by a semisimple $\mathbb{Q}$-algebra, and obtain Fr\"{o}hlich's non-commutative id\`{e}le class group in
an analogous fashion, modulo the reduced norm one elements. Even in the number
field case our proof is simpler than the existing one, and based on the
localization theorem for percolating subcategories. Finally, using class field theory as input, we interpret Hilbert's reciprocity law (as well as a noncommutative variant) in terms of our results.
\end{abstract}
\maketitle

 \tableofcontents          




\section{Introduction}

Let $F$ be a number field and $\mathsf{LCA}_{F}$ the category of locally
compact topological $F$-vector spaces, that is: objects are topological $F$-vector spaces with a locally compact topology and morphisms are continuous $F$-linear maps. 
Clausen \cite{clausen} had predicted that%
\[
K_{1}(\mathsf{LCA}_{F})=\text{Chevalley's id\`{e}le class group,}%
\]
so that the first $K$-group gives the automorphic side of the id\`{e}le
formulation of global class field theory in the number field situation. This
is quite remarkable in that it requires no manual modifications in order to get
the infinite places into the picture, which is frequently
needed when using cohomological, cycle-theoretic or $K$-theoretic approaches
in arithmetic applications.

This picture also finds the correct automorphic object for local class field
theory or finite fields, every time just using $\mathsf{LCA}_{F}$ for the
respective type of field. In this paper we focus exclusively on the hardest part, the
number field case. The original predictions were confirmed in
\cite{kthyartin}, along with a determination of the entire $K$-theory spectrum
of $\mathsf{LCA}_{F}$. In this paper, we replace the number field $F$ by an
arbitrary finite-dimensional semisimple $\mathbb{Q}$-algebra $A$ and $\mathsf{LCA}_{A}$ is analogously defined to consist of locally compact $A$-modules.

We find generalizations of the corresponding results in \cite{kthyartin}, and
with simpler proofs. Let $K\colon\operatorname*{Cat}_{\infty}^{\operatorname*{ex}%
}\rightarrow\mathsf{A}$ be a localizing invariant in the sense of \cite{MR3070515} with values in a stable presentable $\infty$-category $\mathsf{A}$
which commutes with countable direct products of categories, e.g.,
non-connective algebraic $K$-theory.

\begin{theoremann}\label{theorem:MainIntroduction}
Let $A$ be a finite-dimensional semisimple $\mathbb{Q}$-algebra. Then there is
a fiber sequence%
\[
K(A)\longrightarrow K\mathbb{(}\mathsf{LCA}_{A,ab})\longrightarrow
K(\mathsf{LCA}_{A})\text{,}%
\]
where the middle category $\mathsf{LCA}_{A,ab}$ (the so-called
\textquotedblleft adelic blocks\textquotedblright) can be characterized in any
of the following equivalent ways:

\begin{enumerate}
\item It is the category of locally compact $A$-modules of type $\mathbb{A}%
$\ in the sense of Hoffmann--Spitzweck \cite[Definition 2.1]{MR2329311}.

\item It is the full subcategory of objects in $\mathsf{LCA}_{A}$ which are
simultaneously injective and projective.

\item It is a categorical restricted product,%
\[
\mathsf{LCA}_{A,ab}\simeq\left.  \prod\nolimits_{p}^{\prime}\right.
\operatorname*{proj}(A_{p})\text{,}%
\]
where the index $p$ runs over all prime numbers as well as $p=\mathbb{R}$, with the
meaning $A_{\mathbb{R}}=A\otimes_{\mathbb{Q}}\mathbb{R}$ and $A_{p}=A\otimes
_{\mathbb{Q}}\mathbb{Q}_{p}$, respectively. A precise definition of the middle category is
to employ $\mathsf{J}_{A}^{(\infty)}$ from Definition \ref{def_A} below.
\end{enumerate}
The map $K(A)\to K (\mathsf{LCA}_{A,ab})$ is induced by the functor $- \otimes_A \bA$ where $\bA$ is the ring of ad\`{e}les of $A$.
\end{theoremann}

The characterization (3) shows most clearly how the idea of the restricted
product in Chevalley's id\`{e}les is lifted to a categorical level by using
locally compact modules. The characterization in (2) is new, has no previous
counterpart in the literature, and affirmatively settles \cite[Conjecture
1]{kthyartin}. The fiber sequence will be Theorem \ref{thm_FiberSeq}. The
different characterizations of $\mathsf{LCA}_{A,ab}$ follow from the
definition in case of (1), from Theorem
\ref{thm_AdelicBlocksAreProjectiveAndInjective} for (2), and from Proposition
\ref{prop_IdentifyJ} for (3).

To establish the fiber sequence in the above theorem, we follow the method of \cite{paper1} closely.  Specifically, we perform two subsequent localizations of the category $\mathsf{LCA}_{A}$.  In the first step, we consider the quotient by the category $\mathsf{LCA}_{A,com}$ of compact $A$-modules. In the second step, we further quotient by the (image of the) category $\mathsf{LCA}_{A,ab}$ of adelic blocks.  After performing these quotient constructions, one is left with the abelian category $\Mod A / \operatorname{mod} A$ and obtains the required fiber sequence.  This will be done in \S\ref{KTheoryComputations}.

The quotient $Q\colon \mathsf{LCA}_{A} \to \mathsf{LCA}_{A} / \mathsf{LCA}_{A, com}$ is obtained as a localization of $\mathsf{LCA}_{A}$ at the following class of morphisms: all inflations (deflations) with cokernel (kernel) in $\mathsf{LCA}_{A, com}.$   It is not clear that this localization can be equipped with an exact structure such that the localization functor $Q$ is exact, but it follows from \cite{hr} that $\mathsf{LCA}_{A} / \mathsf{LCA}_{A, com}$ has a natural one-sided exact structure (in the sense of \cite{BazzoniCrivei13}) for which $Q$ is exact.  We will recall the necessary setup for one-sided exact categories, and their place in the theory of localizations of exact categories in \S\ref{Section:OneSided}.

On the one hand, one-sided exact categories are a generalization of exact categories.  On the other hand, a one-sided exact category $\EE$ can be completed $2$-canonically to an exact category $\ex{\EE}$, called the \emph{exact hull} of $\EE$.  This completion yields an equivalence $\Dbinf(\EE) \to \Dbinf(\ex{\EE})$ of the derived $\infty$-categories, see \cite{hr2}, and as such, the flexibility added by considering the one-sided exact setting comes at no cost to applications concerning localizing invariants.

For the next quotient, namely the quotient of $\ex{[\mathsf{LCA}_{A} / \mathsf{LCA}_{A, com}]}$ by the image of the category $\mathsf{LCA}_{A,ab}$ of adelic blocks, we use the following attractive property of $\ex{[\mathsf{LCA}_{A} / \mathsf{LCA}_{A, com}]}$: all morphisms have cokernels and all cokernels are deflations.  As is shown in \S\ref{Section:OneSided}, this property holds for $\ex{[\mathsf{LCA}_{A} / \mathsf{LCA}_{A, com}]}$, as it holds for $\mathsf{LCA}_{A}$ (since the latter is quasi-abelian), and is preserved by both the localization and the exact hull.

The primary application is using non-connective $K$-theory as the localizing invariant $K$. It is easily shown to agree with ordinary Quillen $K$-theory (i.e., connective $K$-theory) for $\mathsf{LCA}_A$. Following this path, we recover the non-commutative analogue of the id\`{e}le class group due to Fr\"{o}hlich \cite{MR0376619}. That is, we rediscover a group which Fr\"{o}hlich defined manually in 1975 and without any reference to $K$-theory, and every part of Fr\"{o}hlich's definition fits perfectly (including being a restricted product and the infinite places).

\begin{theoremann}
Let $A$ be a finite-dimensional semi-simple $\mathbb{Q}$-algebra.

\begin{enumerate}
\item There is a natural isomorphism%
\[
K_{1}(\mathsf{LCA}_{A,ab})\overset{\sim}{\longrightarrow}\frac{J(A)}{J^{1}%
(A)}\text{,}%
\]
where $J(A)$ is Fr\"{o}hlich's id\`{e}le class group and $J^{1}(A)$ the
subgroup of reduced norm one elements.

\item There is a natural isomorphism%
\[
K_{1}(\mathsf{LCA}_{A})\overset{\sim}{\longrightarrow}\frac{J(A)}%
{J^{1}(A)\cdot\operatorname*{im}A^{\times}}\text{,}%
\]
where the units $A^{\times}$ are diagonally mapped to the id\`{e}les.
\end{enumerate}
\end{theoremann}

This will be Theorem \ref{thm_ComputeK1}. Not only does this re-establish
Clausen's prediction, it also proves the natural non-commutative analogue.

We can also describe $K_2$ in a precise way. Unlike the previous results, this relies on tools from class field theory.
\begin{theoremann}
\label{thm_B2}
Let $A$ be a finite-dimensional semisimple $\mathbb{Q}$-algebra. Write $\zeta(A)$ for its center and $(-)^{\wedge }$ for profinite completion.
\begin{enumerate}
\item There is a natural isomorphism%
\[
K_{2}(\mathsf{LCA}_{A,ab})\cong\bigoplus_{v}K_{2}(\zeta(A)_{v})\text{,}%
\]
where $v$ runs through the places of the number field $\zeta(A)$.

\item If $A$ is commutative (i.e. a number field),%
\[
K_{2}(\mathsf{LCA}_{A})^{\wedge }\cong\mu(A)\text{,}%
\]
where $\mu(-)$ denotes the group of roots of unity.

\item If Conjecture \ref{Conj_MerkurjevSuslin} of Merkurjev--Suslin holds,%
\[
K_{2}(\mathsf{LCA}_{A})^{\wedge }\cong\mu(\zeta(A))\text{,}%
\]
i.e. no condition on commutativity as in (2) is needed.
\end{enumerate}
\end{theoremann}

The individual parts of the above theorem are shown in the last three sections of the paper. These results also allow us to phrase Moore's formulation of the Hilbert Reciprocity Law \cite{MR244258},
\[
K_{2}(A)\longrightarrow\bigoplus_{v\enskip\mathrm{noncomplex}}\mu(A_{v}) \longrightarrow \mu(A)\longrightarrow0
\]
in terms of the fiber sequence of our Theorem \ref{theorem:MainIntroduction}. Let us stress however that our work does not give an independent proof of Hilbert Reciprocity, because the proof of Theorem \ref{thm_B2} relies on class field theory itself. For this reason, it would be very interesting if one could prove Theorem \ref{thm_B2} using different technology.
 Moreover, we conjecture the following.

\begin{conjecture}
Suppose $A$ is a finite-dimensional simple $\mathbb{Q}$-algebra
with center $F$. Then%
\[
K_{2}(\mathsf{LCA}_{A})\cong K_{2}(\mathsf{LCA}_{F})\text{.}%
\]

\end{conjecture}

This conjecture would follow from the aforementioned conjecture of Merkurjev--Suslin (see \ref{Conj_MerkurjevSuslin} for its statement).

\begin{acknowledgement}
{We thank P. Arndt, B. Chow, B. Kahn, U. Rehmann, A. Vishik for correspondence. Thanks go to the Oberwolfach Seminar on $\operatorname{TC}$ and arithmetic in October 2019.}  The last author is currently a postdoctoral researcher at FWO (12.M33.16N).
\end{acknowledgement}

\section{Structure theory}

\textit{Conventions.} Algebras are associative and unital, but not necessarily
commutative. Ring homomorphisms preserve the unit. The notation
$\operatorname*{mod}(R)$ resp. $\operatorname*{proj}(R)$ refers to the
categories of finitely generated right $R$-modules, resp. finitely generated
projective right modules. Unless otherwise stated, all modules will be right modules.

Suppose $A$ is a finite-dimensional semisimple $\mathbb{Q}$-algebra.

\begin{definition}
The category $\mathsf{LCA}_{A}$ has

\begin{enumerate}
\item as objects locally compact right $A$-modules, where $A$ is given the
discrete topology,

\item as morphisms all continuous right $A$-module homomorphisms.
\end{enumerate}

An exact structure is given by declaring closed injections inflations and open
surjections deflations.
\end{definition}

The proof that this describes a quasi-abelian category with its standard exact
structure can be adapted verbatim from Hoffmann--Spitzweck \cite{MR2329311},
who pioneered this kind of consideration and studied $\mathsf{LCA}%
_{\mathbb{Z}}$.

\begin{remark}
\label{rmk_z1}Alternatively, one can define $\mathsf{LCA}_{A}$ as follows.  Regard the ring $A$ as a category with one object with itself as its
endomorphism ring. Then let%
\[
\mathsf{LCA}_{A}\coloneqq\mathsf{Fun}(A,\mathsf{LCA}_{\mathbb{Z}})\text{,}%
\]
i.e.~we consider the functor category to plain locally compact abelian groups.
Equip the functor category with the pointwise exact structure from
$\mathsf{LCA}_{\mathbb{Z}}$. Kernels and cokernels are computed pointwise, and
hence $\mathsf{LCA}_{A}$ is quasi-abelian because $\mathsf{LCA}_{\mathbb{Z}}$ is.
\end{remark}

\begin{example}
We stress that the condition to be a topological $A$-module is a non-trivial
constraint even though $A$ carries the discrete topology. For example, not
every $\mathbb{Q}$-vector space, equipped with some locally compact topology
on its additive group, is a locally compact $\mathbb{Q}$-module. Indeed, from
the viewpoint of Remark \ref{rmk_z1}, an object is given by some
$M\in\mathsf{LCA}_{\mathbb{Z}}$ along with a ring homomorphism $A\rightarrow
\operatorname*{End}_{\mathsf{LCA}_{\mathbb{Z}}}(M)$ and even if $A=\mathbb{Q}%
$, the induced endomorphisms need not be continuous on $M$.
\end{example}

Pontryagin duality induces an exact equivalence of exact categories%
\[
\mathsf{LCA}_{A}^{op}\overset{\sim}{\longrightarrow}\mathsf{LCA}_{A^{op}%
}\text{.}%
\]
If $A$ is commutative, this renders $\mathsf{LCA}_{A}$ an exact category with
duality. We first set up some basic structure results about the objects of the
category $\mathsf{LCA}_{A}$. The main definition is the following.

\begin{definition}\label{Definition:QuasiAdelicBlock}
We call a module $Q\in\mathsf{LCA}_{A}$ a \emph{quasi-adelic block} if it can
be written as%
\begin{equation}
Q\simeq V\oplus H\qquad\text{with}\qquad H=\bigcup_{n\geq1}\frac{1}%
{n}C\text{,}\label{l_Eq_4}%
\end{equation}
where $V$ is a vector $A$-module and $C$ a compact clopen $\mathfrak{A}%
$-submodule of $H$, where $\mathfrak{A}\subset A$ is any $\mathbb{Z}$-order.
We call it an \emph{adelic block} if additionally $C$ can be chosen such that%
\[
\bigcap_{n\geq1}nC=0\text{.}%
\]
We write $\mathsf{LCA}_{A,qab}$ for the full subcategory of quasi-adelic
blocks, and $\mathsf{LCA}_{A,ab}$ for the one of adelic blocks.
\end{definition}

This definition generalizes \cite[Definition 2.3]{kthyartin}.

\begin{remark}
If $\mathfrak{A},\mathfrak{A}^{\prime}\subset A$ are $\mathbb{Z}$-orders, then
there exists some $N\geq1$ such that $\mathbb{Z}[\frac{1}{N}]\cdot
\mathfrak{A}=\mathbb{Z}[\frac{1}{N}]\cdot\mathfrak{A}^{\prime}$, so there is
no difference whether in the above definition we pick one order $\mathfrak{A}
$ once and for all, or allow any order (the way the definition is stated).
\end{remark}

\begin{lemma}
\label{Lemma_ForgetfulFuncIsFullyFaithful}Let $\mathfrak{A}\subset A$ be any
order. Then the forgetful functors $\mathsf{LCA}_{A}\longrightarrow
\mathsf{LCA}_{\mathfrak{A}}$ and $\mathsf{LCA}_{A}\longrightarrow
\mathsf{LCA}_{\mathbb{Q}}$ are fully faithful and reflect exactness.
\end{lemma}

\begin{proof}
Clear.
\end{proof}

\begin{lemma}
\label{lemma_CharacterizeInjectivesAndProjectives}We record the following observations:

\begin{enumerate}
\item Every discrete right $A$-module is a projective object in $\mathsf{LCA}%
_{A}$.

\item Every compact right $A$-module is an injective object in $\mathsf{LCA}%
_{A}$. Moreover, they are connected.

\item Every vector $A$-module is both injective and projective in
$\mathsf{LCA}_{A}$. Moreover, they are connected.
\end{enumerate}
\end{lemma}

\begin{proof}
(1) Since $A$ is semisimple, every (right) $A$-module $P$ is projective, so
for any surjection $M\twoheadrightarrow P$ a splitting exists as an $A$-module
homomorphism. This gives a splitting also in $\mathsf{LCA}_{A}$ as soon as it
is continuous, but since $P$ carries the discrete topology, this is
automatically true. (2) If $I$ is a compact right $A$-module, its Pontryagin
dual $I^{\vee}$ is a discrete right $A^{op}$-module. By (1) applied to
$I^{\vee}$ for the semisimple algebra $A^{op}$, it follows that $I^{\vee}$ is
a projective object in $\mathsf{LCA}_{A^{op}}$, and since Pontryagin duality
is an exact functor, it follows that $I=I^{\vee\vee}$ is injective in
$\mathsf{LCA}_{A}$. Finally, a compact LCA group is connected if and only if
its dual its torsion-free; but all $A$-modules are $\mathbb{Q}$-vector spaces
and thus torsion-free; see \cite[Lemma 2.21]{kthyartin}. (3) Pick any order
$\mathfrak{A}\subset A$. Regard the vector module as a vector $\mathfrak{A}%
$-module. Then it is projective and injective by \cite[Proposition~8.1]{etnclca} in
$\mathsf{LCA}_{\mathfrak{A}}$. The relevant splitting maps thus exist in
$\mathsf{LCA}_{\mathfrak{A}}$, then apply Lemma
\ref{Lemma_ForgetfulFuncIsFullyFaithful}.
\end{proof}


\begin{definition}
We call an object $M\in\mathsf{LCA}_{A}$ \emph{vector-free} if it does not
have a non-zero vector $A$-module as a subobject (equivalently as a quotient,
or equivalently as a direct summand).
\end{definition}

The equivalence of these characterizations follows from the fact that vector
$A$-modules are both injective and projective by Lemma
\ref{lemma_CharacterizeInjectivesAndProjectives}.

\begin{lemma}
\label{lemma_VectorFreePieceOfAdelicBlockIsTopTorsion}Any vector-free adelic block $Q$ is a topological torsion LCA group. In particular, $Q$ is totally disconnected.
\end{lemma}

\begin{proof}
By definition, $Q=\bigcup_{n\geq 1}\frac{1}{n}C$ and $\bigcap_{n\geq1}nC=0$ for $C$ a compact clopen $\mathfrak{A}$-submodule of $Q$. Thus,
the intersection of all open $\mathbb{Z}$-submodules is also zero. It follows
that $Q$ is totally disconnected (\cite[Theorems~7.8 and 7.3]{MR551496}). Combined with \cite[Theorem~7.7]{MR551496}, it follows that $Q$
admits arbitrarily small compact clopen subgroups such that the quotients%
\[
Q/nC=\bigcup_{N\geq1}\frac{1}{N}(C/nC)
\]
are discrete torsion groups. The topology is the discrete one since $nC$ is
open in $C$ since multiplication with $\mathbb{Q}^{\times}$ acts as
homeomorphisms. Thus, by \cite[Theorem~3.5]{MR637201} the claim follows.
\end{proof}

\begin{lemma}\label{lemma_NoMapsFromQAToDiscrete}
	Write $\mathsf{LCA}_{A,dis}$  and $\mathsf{LCA}_{A,com}$ for the full subcategories of $\mathsf{LCA}_{A}$ of discrete $A$-modules and of compact $A$-modules. The following hold:
	\begin{enumerate}
		\item $\Hom_{\mathsf{LCA}_A}(\mathsf{LCA}_{A,qab},\mathsf{LCA}_{A,dis})=0;$
		\item $\Hom_{\mathsf{LCA}_A}(\mathsf{LCA}_{A,com},\mathsf{LCA}_{A,ab})=0.$
	\end{enumerate}
\end{lemma}

\begin{proof}
	\begin{enumerate}
		\item Let $h\colon Q\to D$ be a map with $Q\in \mathsf{LCA}_{A,qab}$ and $D\in \mathsf{LCA}_{A,dis}$. Let $Q\cong V\oplus H$ and $H=\bigcup_{n\geq1}\frac{1}{n}C$ be as in Definition \ref{Definition:QuasiAdelicBlock}. By Lemma \ref{lemma_CharacterizeInjectivesAndProjectives}, $h(V)$ is connected and thus $h(V)=0$ as $D$ is discrete. Similarly, $h(C)$ is compact as $C$ is compact and thus $h(C)$ is finite as $D$ is discrete. It follows that $h(H)$ is a torsion group and an $A$-module, thus $h(H)=0$.
		\item	Let $h\colon K\to W$ be a map with $K\in \mathsf{LCA}_{A,com}$ and $W\in \mathsf{LCA}_{A,ab}$. By Lemma \ref{lemma_CharacterizeInjectivesAndProjectives}, $K$ is connected. By Definition \ref{Definition:QuasiAdelicBlock} and Lemma \ref{lemma_VectorFreePieceOfAdelicBlockIsTopTorsion}, the result follows.\qedhere
	\end{enumerate}
\end{proof}

\begin{proposition}
\label{prop_Structure}For all $M\in\mathsf{LCA}_{A}$, there exists a split conflation
\[Q\rightarrowtail M \twoheadrightarrow D,\]
unique up to isomorphism, where $Q$ is a quasi-adelic $A$-module and $D$ is a discrete $A$-module. In other words, $(\mathsf{LCA}_{A,qab},\mathsf{LCA}_{A,dis})$ is a split torsion pair in $\mathsf{LCA}_{A}$.
\end{proposition}

\begin{proof}
(This generalizes and simplifies \cite[Theorem 2.7]{kthyartin}) Let
$\mathfrak{A}\subset A$ be any order. By the structure theorem of $\mathsf{LCA}%
_{\mathfrak{A}}$ there exists a conflation%
\begin{equation}
V\oplus C\rightarrowtail M\twoheadrightarrow D^{\prime} \label{l_Eq_A1}%
\end{equation}
with $V$ a vector $\mathfrak{A}$-module, $C$ a compact clopen $\mathfrak{A}$-module in $M$ and
$D^{\prime}$ a discrete $\mathfrak{A}$-module. Define%
\[
Q\coloneqq\bigcup_{n\geq1}\frac{1}{n}\left(  V\oplus C\right)
\]
inside $M$. As Sequence \ref{l_Eq_A1} is exact in $\mathsf{LCA}_{\mathfrak{A}}$ and $D^{\prime}$ discrete, it
follows that $V\oplus C$ is open in $M$. As the action of $\mathbb{Q}\subset A$ is
continuous, $\mathbb{Q}^{\times}$ acts through homeomorphisms and therefore
each $\frac{1}{n}(V\oplus C)$ is also open. Thus, $Q$ is a union of open sets and
therefore itself an open (and thus clopen) $\mathfrak{A}$-module of $M$. As
$\mathbb{Q}\cdot\mathfrak{A}=A$, it follows that $Q$ is even an $A$-submodule.
Since the scalar action of $A$ on $M$ is continuous, the restriction to the
submodule $Q$ is also continuous. Thus, $Q\in\mathsf{LCA}_{A}$. Moreover, we
clearly have the direct sum decomposition%
\[
Q=V\oplus H\qquad\text{with}\qquad H=\bigcup_{n\geq1}\frac{1}{n}C
\]
in $\mathsf{LCA}_{A}$ since $V=\bigcup_{n\geq1}\frac{1}{n}V$ is already an $A
$-module and $V\cap\frac{1}{n}C=0$ for all $n\geq1$ because the intersection
is a compact subset of a real vector space and thus trivial. As $Q$ is a clopen
$A$-submodule of $M$, it is clear that the quotient exists in $\mathsf{LCA}%
_{A}$ and carries the discrete topology. Thus, we have a conflation%
\[
Q\rightarrowtail M\twoheadrightarrow D\text{.}%
\]
By Lemma \ref{lemma_CharacterizeInjectivesAndProjectives} the
discrete $A$-module $D$ is a projective object, so this conflation splits.

It now follows from Lemma \ref{lemma_NoMapsFromQAToDiscrete} and the fact that every $M$ fits into a (split) conflation $Q\rightarrowtail M \twoheadrightarrow D$ that $(\mathsf{LCA}_{A,qab},\mathsf{LCA}_{A,dis})$ is a split torsion theory.
\end{proof}

\begin{lemma}
\label{lemma_DecompQAB}For all $Q\in\mathsf{LCA}_{A,qab}$, there exists a split conflation 
\[K\rightarrowtail Q\twoheadrightarrow W,\] unique up to isomorphism, where $K$ is a compact $A$-module and $W$ is an adelic block. In other words, $(\mathsf{LCA}_{A,com}, \mathsf{LCA}_{A,ab})$ is a split torsion theory in $\mathsf{LCA}_{A,qab}$.
\end{lemma}

\begin{proof}
(This generalizes and simplifies \cite[Proposition~2.23]{kthyartin}) If $Q$ has a
vector module summand $V$, we can write $Q=Q^{\prime}\oplus V$, where
$Q^{\prime}$ is vector-free. If we prove our claim for $Q^{\prime}$, we
prove it in general because if we have $Q^{\prime}\cong K^{\prime}\oplus
W^{\prime}$, we get $Q\cong K^{\prime}\oplus(W^{\prime}\oplus V)$ and the
second summand in brackets is an adelic block. Thus, we may assume without
loss of generality that $Q$ is vector-free. We shall first
construct a conflation%
\begin{equation}
K\rightarrowtail Q\twoheadrightarrow W \label{l_Eq_A2}%
\end{equation}
where $K$ is a compact $A$-module and $W$ an adelic block. To this end, write%
\[
Q=H\qquad\text{with}\qquad H=\bigcup_{n\geq1}\frac{1}{n}C
\]
with $C$ a compact clopen $\mathfrak{A}$-submodule of $H$ as in Definition \ref{Definition:QuasiAdelicBlock}. Define%
\begin{equation}
K\coloneqq\bigcap_{n\geq1}nC\text{.} \label{lmmx1}%
\end{equation}
Since $C$ is compact, each $nC$ is compact (image under the multiplication by
$n$ map), and thus $K$ is an intersection of closed sets and therefore closed.
Moreover, since each $nC$ is an $\mathfrak{A}$-module, so is $K$. Finally, any
$k\in K$ admits elements $c_{N}\in C$ such that $k=N\cdot c_{N}$ for all
$N\in\mathbb{Z}_{\geq1}$. Thus,%
\[
\frac{1}{n}k=\frac{1}{n}Nnc_{Nn}=Nc_{Nn}%
\]
holds for all $N$. This shows that $\frac{1}{n}k\in K$, and by $\mathbb{Q}%
\cdot\mathfrak{A}=A$ it follows that $K$ is an $A$-submodule of $H$. Further,
since the scalar action of $A$ is continuous on $Q$, it remains so on $K$.
Thus, $K\in\mathsf{LCA}_{A}$ and since $K\subseteq C$, it is a compact
$A$-module. The quotient $\pi\colon Q\deflation Q/K(\eqqcolon W)$ is an open map. We now show that $W$ is an adelic block. As $C$ is a clopen compact in $Q$, the image $\pi(C)$ is clopen compact in $W$.  As $\pi$ is surjective, we have%
\[
W=\bigcup_{n\geq1}\frac{1}{n}\pi(C)\text{.}%
\]
We need to show $\bigcap_{n\geq1}n\pi(C)=0$ in $W$. Suppose $x\in Q$ such that
in the quotient $W$ we have $q(x)=\bigcap_{n\geq1}n\pi(C)$. Thus, for any
$n\geq1$ we find $x_{n}\in C$ such that%
\[
x=nx_{n}+k_{n}%
\]
with $k_{n}\in K$. Write $k_{n}=nh_{n}$ for some $h_{n}\in C$, which is
possible by Equation \eqref{lmmx1}. Hence, $x=n(x_{n}+h_{n})$ with $x_{n}%
,h_{n}\in C$, so $x\in\bigcap_{n\geq1}nC=K$, i.e. $x$ maps to zero in $W $.
Thus, $K$ is an adelic block. Since $K$ is a compact $A$-module, it
is injective by Lemma \ref{lemma_CharacterizeInjectivesAndProjectives}, so the
conflation in Equation \eqref{l_Eq_A2} splits. Finally, by Lemma \ref{lemma_NoMapsFromQAToDiscrete}, $\Hom_{\mathsf{LCA}_A}(\mathsf{LCA}_{A,com},\mathsf{LCA}_{A,ab})=0$. The result follows.
\end{proof}

\begin{corollary}
\label{cor_Structure}Suppose $M\in\mathsf{LCA}_{A}$. Then the canonical
filtration of Hoffmann--Spitzweck \cite[Proposition~2.2]{MR2329311} lifts to a
canonical filtration%
\[
M_{\mathbb{S}^{1}}\rightarrowtail F_{\mathbb{Z}}M\rightarrowtail M
\]
in $\mathsf{LCA}_{A}$. Here $M_{\mathbb{S}^{1}}$ is a compact $A$-module,
$M_{\mathbb{A}}\coloneqq F_{\mathbb{Z}}M/M_{\mathbb{S}^{1}}$ is an adelic block and
$M_{\mathbb{Z}}\coloneqq M/F_{\mathbb{Z}}M$ is a discrete vector space. Each step of
the filtration splits, and thus yields a (non-canonical) direct sum
decomposition%
\begin{equation}
M=K\oplus X\oplus D \label{l_c1}%
\end{equation}
such that $X$ is an adelic block, $K$ a compact $A$-module and $D$ a discrete
$A$-module.
\end{corollary}

Let us also recall that the Hoffmann--Spitzweck filtration is functorial: any
continuous morphisms respects the filtration \cite[Proposition~2.2]{MR2329311}.

\begin{proof}
(This generalizes and simplifies \cite[Theorem 3.10]{kthyartin})   Combining Proposition
\ref{prop_Structure} and Lemma \ref{lemma_DecompQAB} gives the second claim,
Equation \eqref{l_c1}. Now combine this with \cite[Proposition~2.2]{MR2329311} and
observe that the topological properties of $K,X,D$ pin down how they have to
occur in the canonical filtration. The rough idea is as follows: let $M^{0}$
be the connected component of $M$. This is canonically determined. Using%
\[
M=K\oplus X\oplus D\text{,}%
\]
we see that it can only stem from $K\oplus V$, where $V$ is the vector module
summand in $X$, since $D$ is discrete and adelic blocks aside from the vector
module are totally disconnected (Lemma
\ref{lemma_VectorFreePieceOfAdelicBlockIsTopTorsion}). Thus, $K\oplus V$ is
canonically determined. The Pontryagin dual is $K^{\vee}\oplus V^{\vee}$,
where $V^{\vee}$ is connected and $K^{\vee}$ discrete. So, again the connected
component uniquely pins down a subobject; and thus allows us to canonically
characterize $V$ as a quotient of $M^{0}$. This settles $M_{\mathbb{S}^{1}}$.
\end{proof}

\begin{remark}
It is not true that the entire category $\mathsf{LCA}_{A,qab}$ is a direct sum
of $\mathsf{LCA}_{A,com}$ and $\mathsf{LCA}_{A,ab}$ as there can be non-zero morphisms from adelic blocks
to compact $A$-modules. To see this, consider the dense image morphism%
\[
\mathbb{Q}\longrightarrow\mathbb{Q}_{p}%
\]
in $\mathsf{LCA}_{\mathbb{Q}}$ and take its Pontryagin dual, noting that
$\mathbb{Q}_{p}^{\vee}\simeq\mathbb{Q}_{p}$ is self-dual, but the dual of
$\mathbb{Q}$ is a compact $\mathbb{Q}$-vector space.
\end{remark}

\begin{lemma}
\label{lemma_CharQABFD}For an object $M\in\mathsf{LCA}_{A}$ its discrete piece
$M_{\mathbb{Z}}$ in the Hoffmann--Spitzweck filtration is a finitely generated
$A$-module if and only if $\operatorname*{Hom}\nolimits_{\mathsf{LCA}_{A}%
}(M,A_{\mathbb{R}})$ is a finite-dimensional real vector space.
\end{lemma}

As the output object $A_{\mathbb{R}}$ has a $\mathbb{R}$-vector space
structure, this $\operatorname*{Hom}$-group canonically comes with an enriched
structure as an $\mathbb{R}$-vector space itself.

\begin{proof}
By Corollary
\ref{cor_Structure} we have%
\[
\operatorname*{Hom}(M,A_{\mathbb{R}})=\operatorname*{Hom}(K,A_{\mathbb{R}%
})\oplus\operatorname*{Hom}(X,A_{\mathbb{R}})\oplus\operatorname*{Hom}%
(D,A_{\mathbb{R}})
\]
and since morphisms respect the Hoffmann--Spitzweck filtration \cite[Proposition~2.2]{MR2329311}, the first summand is always zero, the second always
finite-dimensional (since the vector-free summand of the adelic block maps
trivially to $A_{\mathbb{R}}$ because the compact $C$ in $H=\bigcup\frac{1}%
{n}C$ as in Equation \eqref{l_Eq_4} must map to zero in $A_{\mathbb{R}}$), so we
deduce that $\operatorname*{Hom}(M,A_{\mathbb{R}})$ is finite-dimensional if
and only if $\operatorname*{Hom}(D,A_{\mathbb{R}})$ is. We have
$\operatorname*{Hom}(A,A_{\mathbb{R}})\cong\mathbb{R}^{n}$ for $n\coloneqq\dim
_{\mathbb{Q}}(A)$ since $A$ is a free rank one module over itself, so an
$A$-module homomorphism is given by assigning an arbitrary value, which since
$A$ is discrete, any such map is automatically continuous. Since $A$ is
semisimple, any indecomposable $D\subseteq A$ correspondingly has
$\operatorname*{Hom}(A,A_{\mathbb{R}})$ a finite-dimensional $\mathbb{R}%
$-vector space, and moreover any discrete $A$-module is a (possibly infinite)
direct sum of indecomposables, $D=\bigoplus_{i\in I}D_{i}$. Thus,
$\operatorname*{Hom}(\bigoplus_{i}D_{i},A_{\mathbb{R}})=\prod_{i}%
\operatorname*{Hom}(D_{i},A_{\mathbb{R}})$ is a finite-dimensional
$\mathbb{R}$-vector space if and only if $I$ is a finite set, i.e. if and only
if $D$ is finitely generated.
\end{proof}

Next, we prove a \textquotedblleft Serre subcategory\textquotedblright-type
statement about whether the discrete piece in the Hoffmann--Spitzweck
filtration is finitely generated as an $A$-module.

\begin{lemma}
\label{lemma_QABFDIsSerreSubcat}Suppose $X^{\prime}\rightarrowtail
X\twoheadrightarrow X^{\prime\prime}$ is a conflation in $\mathsf{LCA}_{A}$.
Then the middle object $X$ in its decomposition of Corollary
\ref{cor_Structure} has finitely generated discrete $A$-module part
$X_{\mathbb{Z}}$, if and only $X_{\mathbb{Z}}^{\prime}$ and $X_{\mathbb{Z}%
}^{\prime\prime} $ are finitely generated $A$-modules.
\end{lemma}

\begin{proof}
(This generalizes and isolates the essence of \cite[Propositions~2.14, 2.17 and Theorem~2.19]{kthyartin}) Since $A_{\mathbb{R}}$ is an injective object by Lemma
\ref{lemma_CharacterizeInjectivesAndProjectives}, the sequence%
\[
\operatorname*{Hom}(X^{\prime\prime},A_{\mathbb{R}})\rightarrowtail
\operatorname*{Hom}(X,A_{\mathbb{R}})\twoheadrightarrow\operatorname*{Hom}%
(X^{\prime},A_{\mathbb{R}})
\]
is exact and now the claim follows from Lemma \ref{lemma_CharQABFD} and the
fact that finitely generated real vector spaces are a Serre subcategory in the
category of all real vector spaces.
\end{proof}

\begin{example}
\label{example_CannotTestZeroDiscretePartViaARMaps}The above lemma is only
true for \textquotedblleft finite generation\textquotedblright, but not in a
stronger form which would test being non-zero. The ad\`{e}le sequence
$A\rightarrowtail\mathbb{A}_{A}\twoheadrightarrow\mathbb{A}_{A}/A$ for example
has $(\mathbb{A}_{A})_{\mathbb{Z}}=0$, yet $A_{\mathbb{Z}}=A$ is non-zero.
\end{example}

\begin{lemma}
\label{lemma_ExtensionFacts}We also obtain the following facts.
\begin{enumerate}
\item If $M\in \mathsf{LCA}_A$ and $\Hom(M,A)=0$, then $M\in \mathsf{LCA}_{A,qab}$. 
\item If $M\twoheadrightarrow M^{\prime\prime}$ is a deflation and $M$
quasi-adelic, then $M^{\prime\prime}$ is quasi-adelic.
\item The subcategories $\mathsf{LCA}_{A,qab}$ and $\mathsf{LCA}_{A,ab}$ of $\mathsf{LCA}_A$ both lie extension-closed.
\item If $M^{\prime}\rightarrowtail M$ is an inflation with $M^{\prime}$
discrete and $M$ quasi-adelic, then $M^{\prime}$ is finitely generated.
\end{enumerate}
\end{lemma}

\begin{proof}
Let $M\in\mathsf{LCA}%
_{A}$ be arbitrary. Then we may decompose $M$ as a direct sum as done in
Corollary \ref{cor_Structure}, giving%
\[
\operatorname*{Hom}(K\oplus X\oplus D,A)=\operatorname*{Hom}(D,A)
\]
because all morphisms from $K$ or $X$ to the discrete module $A$ must be zero
by preservation of the Hoffmann--Spitzweck filtration. As $A$ is semisimple,
any $D\neq0$ is a projective $A$-module, i.e. it occurs as a direct summand of
some direct sum $A^{\oplus\kappa}$. It follows that $D\neq0$ is equivalent to
$\operatorname*{Hom}(D,A)\neq0$.

\begin{enumerate}
	\item Assume that $M$ is quasi-adelic. By Lemma \ref{lemma_NoMapsFromQAToDiscrete}, $\Hom(M,A)=0$ as $A$ is discrete. Conversely, assume that $\Hom(M,A)=0$. As above, $\Hom(M,A)=\Hom(D,A)=0$ implies that $D=0$ and thus $M$ is quasi-adelic.
	\item Let $M'\inflation M\deflation M''$ be a conflation in $\mathsf{LCA}_A$. Applying $\Hom(-,A)$, we obtain the left-exact sequence 
	\[0\to \Hom(M'',A)\to \Hom(M,A)\to \Hom(M',A).\]
	If $M$ is quasi-adelic, $\Hom(M,A)=0$ by (1) and thus $\Hom(M'',A)=0$ as well. Again, by (1), $M''$ is quasi-adelic.
	\item Analogously to (2), one finds that if $M',M''\in \mathsf{LCA}_{A,qab}$, then $M\in \mathsf{LCA}_{A,qab}$ and thus $\mathsf{LCA}_{A,qab}\subseteq \mathsf{LCA}_A$ is extension-closed. If $M',M''\in \mathsf{LCA}_{A,ab}$, then $M\in \mathsf{LCA}_{A,qab}$. Hence $M\cong K\oplus X$, furthermore, by Lemma \ref{lemma_NoMapsFromQAToDiscrete}, $\Hom(K,M'')=0$ and thus $K$ is a direct summand of $M'$. As $M'\in \mathsf{LCA}_{A,ab}$, $K=0$ and thus $M\cong X\in \mathsf{LCA}_{A,ab}$.
	\item This follows from Lemma \ref{lemma_QABFDIsSerreSubcat} by observing that $M_{\mathbb{Z}}=0$ if $M$ is quasi-adelic.\qedhere
\end{enumerate}
\end{proof}


\section{Recollections on the id\`{e}le class group}\label{section:Recollections}

Let $\mathfrak{A}\subset A$ be an arbitrary $\mathbb{Z}$-order. We shall use
standard notation, inspired by the choices in \cite[\S 2]{MR0376619},
\cite[\S 2.7]{MR1884523}. For any prime number $p$ define%
\[
\mathfrak{A}_{p}\coloneqq\mathfrak{A}\otimes_{\mathbb{Z}}\mathbb{Z}_{p}%
\qquad\text{and}\qquad A_{p}\coloneqq A\otimes_{\mathbb{Q}}\mathbb{Q}_{p}\text{,}%
\]
where $\mathbb{Z}_{p}$ denotes the $p$-adic integers and $\mathbb{Q}_{p}$ its
field of fractions, the $p$-adic numbers.

\begin{definition}
We write $\Pl_{F}$ for the set of places of a number field $F$. For simplicity, we also write $\Pl$ in place of $\Pl_{\mathbb{Q}}$. We write $\Plfin_{,F}$ for the finite places. We denote the infinite place of $\mathbb{Q}$, often denoted by $\infty$ in the literature, by $\mathbb{R}$. This increases the compatibility to literature, where the real completion is denoted by $A_{\mathbb{R}}$, in the style of a base change, rather than the less customary $A_{\infty}$ (as for example in \cite{MR1884523}).
\end{definition}

Recall that%
\[
\mathfrak{A}_{\mathbb{R}}=A_{\mathbb{R}}=A\otimes_{\mathbb{Q}}\mathbb{R}%
\text{.}%
\]
As $A$ is a finite-dimensional semisimple $\mathbb{Q}$-algebra, it follows
that $A_{p}$ (for $p$ a prime number) is a finite-dimensional semisimple
$\mathbb{Q}_{p}$-algebra and $\mathfrak{A}_{p}$ is a $\mathbb{Z}_{p}$-order in
it (this is \cite[(11.1), (11.2) and (11.5)]{MR1972204}). We also define%
\[
\widehat{\mathfrak{A}}=\mathfrak{A}\otimes_{\mathbb{Z}}\widehat{\mathbb{Z}%
}=\prod_{p\neq\mathbb{R}}\mathfrak{A}_{p}\qquad\text{and}\qquad\widehat
{A}=\widehat{\mathfrak{A}}\otimes_{\mathbb{Z}}\mathbb{Q}\text{.}%
\]
Here $\widehat{\mathbb{Z}}$ is the profinite completion of the integers. The
alternative characterization follows from the decomposition into topological
$p$-parts, $\widehat{\mathbb{Z}}=\prod\mathbb{Z}_{p}$.

Let us recall the generalization of id\`{e}les to the non-commutative setting.
This follow Fr\"{o}hlich's paper \cite[\S 2]{MR0376619}. The \emph{id\`{e}le
group} is defined to be%
\begin{equation}
J(A)\coloneqq\left\{  (a_{p})_{p}\in\left.  \prod_{p \in \Pl}A_{p}^{\times}\right\vert
a_{p}\in\mathfrak{A}_{p}^{\times}\text{ for all but finitely many places
}p\right\}  \text{.} \label{l_def_idelegroupplain}%
\end{equation}
At the infinite place
$p=\mathbb{R}$ this means $\mathfrak{A}_{\mathbb{R}}=A_{\mathbb{R}}$, so there
the condition is void.

It turns out that this group actually does not depend on the choice of an
order $\mathfrak{A}$ because if $\mathfrak{A}^{\prime}\subset A$ is another
$\mathbb{Z}$-order, we have $\mathfrak{A}_{p}=(\mathfrak{A}^{\prime})_{p}$ for
all but finitely many places $p$. But then in view of Equation
\eqref{l_def_idelegroupplain}, the group $J(A)$ does not change. The group%
\begin{equation}
J^{1}(A)\coloneqq\left\{  (a_{p})_{p}\in J(A)\mid\operatorname*{nr}\nolimits_{A_{p}%
}(a_{p})=1\right\}  \label{l_red_norm_one_ideles}%
\end{equation}
is the subgroup of \emph{reduced norm one} id\`{e}les. We denote the center of
a ring $A$ by $\zeta(A)$.

We
write%
\begin{equation}
\underset{v\,\smallskip\,\smallskip\,\smallskip}{\prod\nolimits^{\prime}}%
K_{n}(A_{v})\coloneqq\left\{  \left.  (\alpha_{v})_{v}\in\prod_{v\in\mathcal{P}%
\ell_{F}}K_{n}(A_{v})\right\vert
\begin{array}
[c]{l}%
\alpha_{v}\in\operatorname*{im}K_{n}(\mathfrak{A}_{v})\\
\text{for all but finitely many }v
\end{array}
\right\}  \label{lcw1}%
\end{equation}
for $\mathfrak{A}\subset A$ any order and $\mathfrak{A}_{v}\coloneqq\mathfrak{A}%
\otimes_{\mathcal{O}_{F}}\mathcal{O}_{F,v}$. The group defined in Equation
\ref{lcw1} is independent of the choice of the order $\mathfrak{A}\subset A$.
Here and everywhere below, we use the notation $\underset{v\,\smallskip
\,\smallskip\,\smallskip}{\prod\nolimits^{\prime}}$ with the above meaning and
not in the sense of a restricted product of topological groups.

\begin{example}
Suppose $A=F$ is a number field and $n=1$. Then each $K_{1}(F_{v}%
)=F_{v}^{\times}$ may also be regarded as a locally compact topological group
with respect to the metric on $F_{v}$ and $\mathcal{O}_{F,v}^{\times}\subseteq
F_{v}^{\times}$ is a compact clopen subgroup. Then the restricted product
$\underset{v\,\smallskip\,\smallskip\,\smallskip}{\prod\nolimits^{\prime}%
}K_{1}(F_{v})$ in the sense of topological groups has the same underlying
abstract group as the group in Equation \eqref{lcw1}. This justifies using the
notation $\underset{v\,\smallskip\,\smallskip\,\smallskip}{\prod
\nolimits^{\prime}}$ also for $K_{n}$ with $n\neq1$.
\end{example}

\section{Adelic blocks}

\subsection{Description as a 2-colimit}

Let $\mathfrak{A}\subset A$ be a maximal order. Such always exists
\cite[Corollary~10.4]{MR1972204}, but it is not unique.

As $\mathfrak{A}$ is maximal, it follows that each $\mathbb{Z}_{p}$-order
$\mathfrak{A}_{p}\subset A_{p}$ is also maximal (this is \cite[Corollary~11.6]{MR1972204}). We shall need the following basic fact.

\begin{lemma}
\label{lemma_StructOfFinGenModulesOverMaxOrder}Every finitely generated
$\mathfrak{A}$-module $M$ (for $\mathfrak{A}\subset A$ a maximal order) has
the shape%
\[
M\cong M_{tor}\oplus P\text{,}%
\]
where $M_{tor}$ is the submodule of $\mathbb{Z}$-torsion elements, and $P$ a
finitely generated projective $\mathfrak{A}$-module. The same is true for
maximal orders $\mathfrak{A}_{p}\subseteq A_{p}$.
\end{lemma}

\begin{proof}
This is standard, but apparently not spelled out in the standard reference
\cite{MR1972204}. Since any morphism sends a $\mathbb{Z}$-torsion element to a
$\mathbb{Z}$-torsion element, this in particular holds for the $\mathfrak{A}%
$-module structure. Thus, the $\mathbb{Z}$-torsion elements $M_{tor}$ form a
right $\mathfrak{A}$-submodule of $M$. As $M$ is finitely generated (even over
$\mathbb{Z}$), it follows that $M_{tor}$ is finite as a set. We get a
conflation $M_{tor}\hookrightarrow M\twoheadrightarrow M/M_{tor}$. Thus, we
are done if we can show that every $\mathbb{Z}$-torsionfree $\mathfrak{A}%
$-module is projective. So, suppose $M$ is torsionfree. We follow \cite[Theorem~10.6]{MR1972204}. Tensoring with $\mathbb{Q}$ gives the embedding
$M\hookrightarrow M\otimes_{\mathbb{Z}}\mathbb{Q}$, the latter is an
$A$-module, thus has some free resolution, giving a surjection%
\[
A^{\oplus n}\twoheadrightarrow M\otimes_{\mathbb{Z}}\mathbb{Q}%
\]
for some integer $n$. As $A$ is semisimple, this sequence must split, giving
$M\subseteq M\otimes_{\mathbb{Z}}\mathbb{Q}\subseteq A^{\oplus n}$. As $M$ is
finitely generated, there exists some $N\geq1$ such that $M\subseteq\frac
{1}{N}\mathfrak{A}^{\oplus n}$ (just using that $\mathfrak{A}\cdot
\mathbb{Q}=A$). Since $\mathfrak{A}$ is a maximal order, it is hereditary
(\cite[Theorem~21.4]{MR1972204}), and thus its submodule $M$ is projective.
\end{proof}

For any exact categories $(\mathsf{C}_{i})_{i\in I}$ and index set $I$ one may
define a product exact category $\prod_{i\in I}\mathsf{C}_{i}$. Its objects
are arrays $X=(X_{i})_{i\in I}$ with $X_{i}\in\mathsf{C}_{i}$ and morphisms
$X\rightarrow X^{\prime}$ are termwise morphisms $X_{i}\rightarrow
X_{i}^{\prime}$ in $\mathsf{C}_{i}$. Conflations are termwise conflations.

\begin{lemma}
\label{lemma_ProductsOfSplitExact}If each exact category $\mathsf{C}_{i}$ is
split exact, so is $\prod_{i\in I}\mathsf{C}_{i}$.
\end{lemma}

\begin{proof}
Given a conflation $(X_{i}^{\prime})\hookrightarrow(X_{i})\twoheadrightarrow
(X_{i}^{\prime\prime})$ in the product category, we have conflations
$X_{i}^{\prime}\hookrightarrow X_{i}\overset{q_{i}}{\twoheadrightarrow}%
X_{i}^{\prime\prime}$ in $\mathsf{C}_{i}$, and if $s_{i}$ is a splitting of
$q_{i}$, the morphism $(q_{i})_{i\in I}$ defines a splitting in the product category.
\end{proof}

\begin{definition}
\label{def_A} Suppose $\mathfrak{A}\subset A$ is a maximal order. Let $S$ be a
finite set of prime numbers. Define%
\[
\mathsf{J}_{A}^{(S)}\coloneqq\operatorname*{proj}(A_{\mathbb{R}})\times\prod_{p\notin
S}\operatorname*{proj}(\mathfrak{A}_{p})\times\prod_{p\in S}%
\operatorname*{proj}(A_{p})\text{.}%
\]
For any inclusion of finite sets $S\subseteq S^{\prime}$ of primes there is an
exact functor $\mathsf{J}^{(S)}\rightarrow\mathsf{J}^{(S^{\prime})}$ induced
from $(-)\mapsto(-)\otimes_{\mathfrak{A}_{p}}A_{p}$ termwise for all primes
$p\in S^{\prime}\setminus S$, and the identity functor for the remaining $p$.
Define%
\begin{equation}
\mathsf{J}_{A}^{(\infty)}\coloneqq\left.  \underset{S}{\operatorname*{colim}}\left.
\mathsf{J}_{A}^{(S)}\right.  \right.  \text{.} \label{lszv1}%
\end{equation}

\end{definition}

This definition is actually independent of the choice of the maximal order,
see Remark \ref{rmk_ChangeMaximalOrder}.

\begin{remark}
[$2$-Colimit]\label{rmk_2Colim}The colimit in Equation \eqref{lszv1} is taken of
a diagram, indexed by a partially ordered set, taking values in the
$2$-category of exact categories. One might call this a $2$-colimit, but it is
also common to merely call this a colimit.
\end{remark}

\begin{proposition}
\label{prop_IdentifyJ}Suppose $\mathfrak{A}\subset A$ is a maximal order. Then
there is an exact equivalence of exact categories%
\[
\mathsf{LCA}_{A,ab}\overset{\sim}{\longrightarrow}\mathsf{J}_{A}^{(\infty)}%
\]
and in particular the definition of $\mathsf{J}_{A}^{(\infty)}$ is independent
of the choice of the order $\mathfrak{A}\subset A$. The functor sends $X$ to
the array $(X_{p})_{p}$ of topological $p$-torsion parts\footnote{See
Definition \ref{def_ModuleStruct_Xp}.}, resp. the vector module part; both on
objects as well as morphisms.
\end{proposition}

Note that the individual categories $\left.  \mathsf{J}_{A}^{(S)}\right.  $ do
not even possess an enrichment in $A$-modules (or even $\mathbb{Q}$-vector
spaces). This only exists in the colimit.

\subsection{Proof of Proposition \ref{prop_IdentifyJ}}

We sketch the strategy. We first `fatten up' the category $\mathsf{LCA}%
_{A,ab}$ to a category $\mathsf{M}_{0}$ with extra data and a forgetful
functor%
\[
\mathsf{M}_{0}\overset{\sim}{\longrightarrow}\mathsf{LCA}_{A,ab}\text{.}%
\]
Then we consider what looks like a subcategory $\mathsf{M}_{1}\subseteq \mathsf{M}_{0}$ at first, but turns out to be an equivalence as well:$\mathsf{M}_{1}\overset{\sim}{\longrightarrow}\mathsf{M}_{0}.$  Next, on $\mathsf{M}_{1}$ one can explicitly write down a functor to
$\mathsf{J}_{A}^{(\infty)}$ and show that it is an equivalence. We obtain a
chain of exact equivalences%
\[
\mathsf{J}_{A}^{(\infty)}\overset{\sim}{\longleftarrow}\mathsf{M}_{1}%
\overset{\sim}{\longrightarrow}\mathsf{M}_{0}\overset{\sim}{\longrightarrow
}\mathsf{LCA}_{A,ab}%
\]
Finally, we check that these functors transform Pontryagin duality into linear duality.

\subsubsection{Step 1}

We define an exact category $\mathsf{M}_{0}$ first. Its objects are pairs%
\[
(X,C)\text{,}%
\]
where $X$ is a vector-free adelic block and $C\subset X$ is a compact clopen
$\mathfrak{A}$-submodule as in Equation \eqref{l_Eq_4}, i.e.%
\begin{equation}
\bigcup\frac{1}{n}C=X\qquad\text{and}\qquad\bigcap\frac{1}{n}C=0\text{.}
\label{lwwaa3}%
\end{equation}
Morphisms $(X,C)\rightarrow(X^{\prime},C^{\prime})$ are just morphisms of
adelic blocks $X\rightarrow X^{\prime}$ in $\mathsf{LCA}_{A,ab}$. Conflations
are such if the adelic blocks form a conflation. In particular, neither for
morphisms nor conflations, there is any dependency on $C$. There is an obvious
functor%
\[
\Psi_{0}\colon\mathsf{M}_{0}\longrightarrow\mathsf{LCA}_{A,ab},\qquad
(X,C)\longmapsto X
\]
forgetting the choice of $C$. This is clearly fully faithful and exact. It is
also essentially surjective since each adelic block possesses a choice of $C$
by definition. Thus, the above functor $\Psi_{0}$ is an exact equivalence of
exact categories,

\begin{definitionp}
\label{def_ModuleStruct_Cp}For a prime number $p$, define the
\emph{topological }$p$\emph{-torsion part}%
\begin{equation}
C_{p}\coloneqq\left\{  x\in C\left\vert \underset{n\rightarrow\infty}{\lim}%
p^{n}x=0\right.  \right\}  \text{.} \label{lwwaa1}%
\end{equation}
This is a closed $\mathfrak{A}$-submodule of $C$. Moreover, it carries a
canonical right $\mathfrak{A}_{p}$-module structure through%
\begin{equation}
x\cdot\alpha=\underset{n\rightarrow\infty}{\lim}(x\cdot\alpha_{n})\text{,}
\label{lwwaa2}%
\end{equation}
where $(\alpha_{n})$ is any sequence with $\alpha_{n}\in\mathfrak{A}$
converging to $\alpha$ in the $p$-adic topology of $\mathfrak{A}_{p}$.
\end{definitionp}

\begin{proof}
For a general LCA group, the topological $p$-torsion part need not be a closed
subgroup, but we know that $C$ is totally disconnected and can invoke
\cite[Lemma~3.8]{MR637201}, settling that $C_{p}$ is closed in $C$. The
limit condition is preserved under the $\mathfrak{A}$-module action. The
$\mathfrak{A}_{p}$-module stems from $\mathfrak{A}_{p}=\mathfrak{A}%
\otimes_{\mathbb{Z}}\mathbb{Z}_{p}$ and $\mathbb{Z}$ being dense in
$\mathbb{Z}_{p}$ in the $p$-adic topology. The existence of the limit is
ensured since in the $p$-adic topology $\left\vert \alpha_{n}-\alpha
\right\vert $ getting smaller means $\alpha_{n}-\alpha\in p^{m}\mathbb{Z}_{p}$
for $m\rightarrow\infty$ as $n\rightarrow\infty$, and therefore $(\alpha
_{n}-\alpha)x\rightarrow0$ thanks to the topological $p$-torsion condition in
Equation \eqref{lwwaa1}.
\end{proof}

\begin{definition}
\label{def_ModuleStruct_Xp}For a prime number $p$, we also define%
\[
X_{p}\coloneqq\left\{  x\in X\left\vert \underset{n\rightarrow\infty}{\lim}%
p^{n}x=0\right.  \right\}  \text{.}%
\]
This is a closed $A$-submodule of $X$. It carries a canonical right $A_{p}%
$-module structure, again defined by Equation \eqref{lwwaa2}.
\end{definition}

This is shown analogously. Again, \cite[Lemma~3.8]{MR637201} settles that
$X_{p}\subseteq X$ is closed. The rest can be deduced in the same way, but by
a simple exercise we also have the description%
\begin{equation}
X_{p}=\bigcup\frac{1}{n}C_{p}\qquad\text{and}\qquad\bigcap\frac{1}{n}C_{p}=0
\label{lwwaa5}%
\end{equation}
analogous to Equation \eqref{lwwaa3}, so one can alternatively reduce this to
the previous considerations about $C_{p}$. Equation \eqref{lwwaa5} also shows
that each $X_{p}$ is a vector-free adelic block.

\begin{lemma}
\label{lemma7}$C_{p}$ is a finitely generated projective $\mathfrak{A}_{p}%
$-module, $X_{p}$ is a finitely generated projective $A_{p}$-module, and%
\begin{equation}
X_{p}=C_{p}\otimes_{\mathfrak{A}_{p}}A_{p} \label{lwwaa6}%
\end{equation}
holds.
\end{lemma}

\begin{proof}
From $C_{p}\subseteq C$ we get $C_{p}/pC_{p}\subseteq C/pC$. However, $p\in
A^{\times}$ is a unit, so multiplication by $p$ is a homeomorphism of $X$ to
itself, so it maps the open $C$ to an open $pC$. Thus, $C/pC$ is compact (as
$C$ is compact) and discrete (as $pC$ is open), and therefore finite.\ It
follows that $C_{p}/pC_{p}$ is finite. Now conclude that $C_{p}$ is finitely
generated by the topological Nakayama Lemma for compact modules, see
\cite[Chapter 5, \S 1]{MR1029028} (instead of $C_{p}$ being finitely
generated, this takes compactness as input). Next, $X_{p}=C_{p}\otimes
_{\mathbb{Z}_{p}}\mathbb{Q}_{p}$ follows directly from Equation \eqref{lwwaa5}.
This yields Equation \eqref{lwwaa6}, and also implies the finite generation
property for $X_{p}$. As $A_{p}$ is semisimple, it is clear that $X_{p}$ is
projective. Next, by Lemma \ref{lemma_StructOfFinGenModulesOverMaxOrder} we
learn that $C_{p}$ is projective, since by $C_{p}\subseteq X_{p}$ it must be
$\mathbb{Z}$-torsionfree.
\end{proof}

\begin{lemma}
\label{lemma5}There is an isomorphism of topological $\mathfrak{A}$-modules,
$\prod_{p}C_{p}\overset{\sim}{\longrightarrow}C$, induced from the respective
inclusions $C_{p}\subseteq C_{p}$.
\end{lemma}

\begin{proof}
\cite[Proposition~3.10]{MR637201}.
\end{proof}

\subsubsection{Step 2}

Next, we define a category $\mathsf{M}_{1}$. Its objects are pairs $(X,C)$ as
before, but this time morphisms $(X,C)\rightarrow(X^{\prime},C^{\prime})$ are
morphisms of adelic blocks%
\[
f\colon X\longrightarrow X^{\prime}\qquad\text{such that}\qquad f(C_{p})\subseteq
C_{p}^{\prime}%
\]
holds for all but finitely many prime numbers $p$. This is clearly closed
under composition. There is an obvious faithful functor%
\[
\Psi_{1}\colon\mathsf{M}_{1}\longrightarrow\mathsf{M}_{0}\text{,}%
\]
as we have just imposed an additional condition on morphisms. We claim that
the functor is also full: let $(X,C)\rightarrow(X^{\prime},C^{\prime})$ be any
morphism in $\mathsf{M}_{0}$, say coming from $f\colon X\rightarrow X^{\prime}$. As
$C$ is compact, the set-theoretic image $f(C)\subseteq X^{\prime}$ is compact.
Further, $X^{\prime}=\bigcup\frac{1}{n}C^{\prime}$ is an open cover, so since
$f(C)$ is compact, a finite subcover suffices. Thus, there exists some
$n\geq1$ such that $f(C)\subseteq\frac{1}{n}C^{\prime}$. In particular,
$f(C_{p})\subseteq\frac{1}{n}C_{p}^{\prime}$. However, $C_{p}^{\prime}$ is an
$\mathbb{Z}_{p}$-module (as it is a $\mathfrak{A}_{p}$-module by the module
structure of Definition \ref{def_ModuleStruct_Cp}). Hence, for all primes $p$
not dividing $n$, $\frac{1}{n}\in\mathbb{Z}_{p}^{\times}$ is a unit, so
$\frac{1}{n}C_{p}^{\prime}=C_{p}^{\prime}$. In other words: except for
finitely many primes we have $f(C_{p})\subseteq C_{p}^{\prime}$, as required.
Since $\Psi_{1}$ is an equivalence of categories, we can transport the exact
structure from $\mathsf{M}_{0}$ to $\mathsf{M}_{1}$: it just amounts to
kernel-cokernel sequences%
\[
(X^{\prime},C^{\prime})\longrightarrow(X,C)\longrightarrow(X^{\prime\prime
},C^{\prime\prime})
\]
being a conflation if and only if $X^{\prime}\hookrightarrow
X\twoheadrightarrow X^{\prime\prime}$ is a conflation of adelic blocks.

\begin{lemma}
\label{lemma6}If $f\colon (X,C)\rightarrow(X^{\prime},C^{\prime})$ is a morphism in
$\mathsf{M}_{1}$, we have $f(X_{p})\subseteq X_{p}^{\prime}$. Moreover,
$f|_{X_{p}}$ is a morphism of $A_{p}$-modules.
\end{lemma}

\begin{proof}
We had seen a little above the lemma that $f(C)\subseteq\frac{1}{n}C^{\prime}%
$.\ Suppose $x\in C_{p}$. Then the element $f(x)$ is also topological
$p$-torsion because $f$ is continuous and therefore $\lim_{n\rightarrow\infty
}p^{n}f(x)=f(\lim_{n\rightarrow\infty}p^{n}x)=0$. If $\ell$ denotes a distinct
prime, let $pr_{\ell}$ denote the projection $C^{\prime}\twoheadrightarrow
C_{p}^{\prime}$ coming from applying Lemma \ref{lemma5} to $C^{\prime}$. It
follows that $pr_{\ell}f(x)$ is also topological $p$-torsion (note that
$\frac{1}{n}C^{\prime}$ is isomorphic to $C^{\prime}$). Hence, it is
topological $I$-torsion in for the ideal $I=(p,\ell)=(1)$. Thus, $pr_{\ell
}f(x)=0$. As this works for all primes $\ell\neq p$, we deduce that
$f(x)\in\frac{1}{n}C_{p}^{\prime}$. Now, for a general $x\in X_{p}$ some
multiple $mx$ lies in $C_{p}$, so our claim follows.\ Note that $f|_{X_{p}%
}$ being an $A_{p}$-module homomorphism just comes from the compatible way how
the module structure is defined in Definition \ref{def_ModuleStruct_Xp}.
\end{proof}

\subsubsection{Step 3}

Now we define a functor%
\[
\Psi_{2}\colon\mathsf{M}_{1}\longrightarrow\left.  \underset{S}%
{\operatorname*{colim}}\left.  \mathsf{J}_{A}^{(S)}\right.  \right.  \text{.}%
\]
On objects, send $(X,C)$ to the array $(C_{p})_{p}$. By Lemma \ref{lemma7}
each $C_{p}$ is a finitely generated projective $\mathfrak{A}_{p}$-module, so
we can take $S=\varnothing$. For any morphism%
\[
(X,C)\longrightarrow(X^{\prime},C^{\prime})
\]
in $\mathsf{M}_{1}$, we know that for all but finitely many primes $p$ we have
$f(C_{p})\subseteq C_{p}^{\prime}$ by the very definition of $\mathsf{M}_{1}$.
Take $S$ to be these primes. Then for all primes $p\notin S$ we have
$f|_{C_{p}}:C_{p}\rightarrow C_{p}^{\prime}$ as a morphism of
$\mathfrak{A}_{p}$-modules by Lemma \ref{lemma6}. For the primes $p\in S$, the
same two lemmata say that $f|_{X_{p}}$ is a morphism of finitely generated
projective $A_{p}$-modules. These constructions are compatible with the
transition morphisms of the $2$-colimit by Equation \eqref{lwwaa6}. It is clear
that this describes a faithful functor. It is also full: suppose we are given
a morphism in $\left.  \mathsf{J}_{A}^{(S)}\right.  $ for a finite set $S$.
For $p\notin S$, every morphism of finitely generated $\mathfrak{A}_{p}%
$-modules is automatically continuous. Thus, it defines a map $C_{p}%
\rightarrow C_{p}^{\prime}$ for each such $p$ and then induces one to the
product%
\begin{equation}
\prod_{p\notin S}C_{p}\rightarrow\prod_{p\notin S}C_{p}^{\prime}%
\text{.}\label{lwwaa8}%
\end{equation}
For the $p\in S$ one can do this similarly: since the source module is
finitely generated, its image is finitely generated, thus lies in $\frac
{1}{p^{N}}C_{p}^{\prime}$ for a suitable $N$; and just use this instead of
$C_{p}^{\prime}$ on the right in Equation \eqref{lwwaa8} for the $p\in S$. By
Lemma \ref{lemma5} and since $X=\bigcup\frac{1}{n}C$ (and similarly for
$X^{\prime}$), this defines a morphism $X\rightarrow X^{\prime}$.

Finally, we claim that $\Psi_{2}$ is essentially surjective: Let $Y$ be an
object in $\left.  \mathsf{J}_{A}^{(S)}\right.  $. Define for all primes $p$,%
\[
\widehat{Y}_{p}\coloneqq Y\otimes_{\mathfrak{A}_{p}}A_{p}\text{,}%
\]
and note that this is well-defined on the $2$-colimit (it does not change
under the transition maps). Each $\widehat{Y}_{p}$ is a finitely generated
projective $A_{p}$-module, having a natural topology as a finite-dimensional
$\mathbb{Q}_{p}$-vector space. Hence, we may choose some finitely generated
torsionfree $\mathfrak{A}_{p}$-submodule $\widehat{C}_{p}$ such that
$\mathbb{Q}\cdot\widehat{C}_{p}=\widehat{Y}_{p}$. Then each $\widehat{Y}_{p}$
is locally compact and $\widehat{C}_{p}$ a compact clopen group. Now take the
restricted product $\prod^{\prime}(\widehat{Y}_{p},\widehat{C}_{p})$, \cite[p.
6, P.14]{MR637201}\footnote{Instead of restricted product, it is called a
\emph{local direct product }in the cited reference.}. This is a locally
compact abelian group and a $\mathfrak{A}$-module through the inclusions
$\mathfrak{A}\subset\mathfrak{A}_{p}$. Its topological $p$-torsion component
is precisely $\widehat{Y}_{p}$.

Thus, $\Psi_{2}$ is an equivalence of categories. It also preserves the exact
structures since by the previous considerations an inflation in $\mathsf{M}%
_{1}$ just corresponds to an injection, and conversely all
these are indeed automatically closed immersions (and analogously for deflations).

This finishes the proof of Proposition \ref{prop_IdentifyJ}.

\begin{remark}
\label{rmk_ChangeMaximalOrder}Suppose we pick a different maximal order
$\mathfrak{A}^{\prime}\subset A$ in Definition \ref{def_A}. Then the
definition does not actually change because the local maximal orders
$\mathfrak{A}_{p}^{\prime}\subset A_{p}$ will be isomorphic, $\mathfrak{A}%
_{p}\simeq\mathfrak{A}_{p}^{\prime}$, \cite[Proposition 3.5]{MR0117252}. If
$A_{p}$ is a division algebra, the maximal order is unique.
\end{remark}

\subsection{Duality\label{subsect_Duality}}

There is an exact functor to the opposite category%
\begin{equation}
(-)_{A}^{\ast}:\mathsf{J}_{A}^{(S)}\longrightarrow\mathsf{J}_{A^{op}}%
^{(S),op}\text{,}\label{lwwaa4}%
\end{equation}
defined on the factor categories by%
\begin{equation}
X_{p}\longmapsto\left\{
\begin{array}
[c]{ll}%
\operatorname*{Hom}\nolimits_{A_{\mathbb{R}}}(X_{\mathbb{R}},A_{\mathbb{R}}) &
\text{for }p=\mathbb{R}\\
\operatorname*{Hom}\nolimits_{\mathfrak{A}_{p}}(X_{p},\mathfrak{A}_{p}) &
\text{for }p\notin S\\
\operatorname*{Hom}\nolimits_{A_{p}}(X_{p},A_{p}) & \text{for }p\in S\text{.}%
\end{array}
\right.  \label{lwwab1}%
\end{equation}
These $\operatorname*{Hom}$-modules use up the right module structures and
thus are then left $A_{\mathbb{R}}$ (resp. $\mathfrak{A}_{p}$, $A_{p}$)
modules. Or, as we shall phrase it, right modules over the opposite ring. Note
that $\mathfrak{A}^{op}$ is a maximal order in $A^{op}$ (if there was a bigger
one, its opposite would be bigger than $\mathfrak{A}$). Thus, all our
considerations above regarding $\mathsf{J}_{A}^{(S)}$ also apply to
$\mathsf{J}_{A^{op}}^{(S)}$. Since the input modules are projective, these
functors are all exact and moreover there is a natural equivalence of functors%
\[
\operatorname*{id}\longrightarrow(-)_{A}^{\ast}\circ\lbrack(-)_{A^{op}}^{\ast
}]^{op}\text{.}%
\]
In particular, $(-)_{A}^{\ast}$ is an exact equivalence of exact categories.

\begin{proposition}
\label{prop_J_Duality}The functors in Equation \eqref{lwwaa4} induce an exact
equivalence of categories in the $2$-colimit%
\[
\mathsf{J}_{A}^{(\infty)}\overset{\sim}{\longrightarrow}\mathsf{J}_{A^{op}%
}^{(\infty),op}%
\]
and under the exact equivalence of Proposition \ref{prop_IdentifyJ} this
identifies with Pontryagin duality on adelic blocks $\mathsf{LCA}%
_{A,ab}\overset{\sim}{\longrightarrow}\mathsf{LCA}_{A^{op},ab}^{op}$.
\end{proposition}

\begin{proof}
For the first claim we only need to check that the duality functor respects
the transition morphisms in the $2$-colimit diagram. We only need to check
this for those $p$ where the transition map is not the identity, i.e., where
$X_{p}$ is a finitely generated $\mathfrak{A}_{p}$-module. However, the
compatibility then amounts to the harmless isomorphism%
\begin{equation}
A_{p}\otimes_{\mathfrak{A}_{p}}\operatorname*{Hom}\nolimits_{\mathfrak{A}_{p}%
}(X_{p},\mathfrak{A}_{p})\cong\operatorname*{Hom}\nolimits_{A_{p}}%
(X_{p}\otimes_{\mathfrak{A}_{p}}A_{p},A_{p})\text{,}\label{lmiopps1a}%
\end{equation}
where we see the relevant transition maps on either side, respectively.
Equation \eqref{lmiopps1a} is an isomorphism, because it is a central
localization by $\mathbb{Z}_{p}\setminus\{0\}$.

The second fact is more surprising since Pontryagin duality is based on $\operatorname*{Hom}%
\nolimits_{\mathsf{LCA}}(-,\mathbb{T})$ instead. To check this, we proceed by
the following reductions:\ For each prime $p$, the equivalence identifies the
full subcategory $\operatorname*{proj}(A_{p})\subset\mathsf{J}_{A}^{(\infty)}$
with the full subcategory of topological $p$-torsion adelic blocks. As we
already have an equivalence of categories, restrict it to
$\operatorname*{proj}(A_{p})$ and its essential image. Being an equivalence of
categories, it suffices to work with $\operatorname*{proj}(A_{p})$ and compare
its natural duality (Equation \eqref{lwwab1}) with the duality transported from
the topological $p$-torsion adelic blocks (i.e. Pontryagin duality) along this
equivalence to $\operatorname*{proj}(A_{p})$. If these are compatible for all
primes $p$ and $p=\mathbb{R}$, then compatibility follows for the $\not 2%
$-colimit. So, first, fix a prime $p$. Since $A_{p}$ is a projective generator
of $\operatorname*{proj}(A_{p})$, it suffices to prove compatibility on this
object. We have%
\[
\operatorname*{Hom}\nolimits_{\mathbb{Q}_{p}}(A_{p},\mathbb{Q}_{p}%
)\cong\operatorname*{Hom}\nolimits_{A_{p}}(A_{p},A_{p})=A_{p}^{\ast}%
\]
by tensoring with $A_{p}$. Now compose any $\varphi$ in the left group with
the composition of maps%
\begin{equation}
\mathbb{Q}_{p}\twoheadrightarrow\mathbb{Q}_{p}/\mathbb{Z}_{p}\subset
\mathbb{Q}/\mathbb{Z}\overset{e}{\longrightarrow}\mathbb{T}\qquad
e(x)\coloneqq\exp(2\pi ix)\text{,}\label{lwwaa7}%
\end{equation}
inducing a map $\operatorname*{Hom}\nolimits_{\mathbb{Q}_{p}}(A_{p}%
,\mathbb{Q}_{p})\rightarrow\operatorname*{Hom}\nolimits_{\mathsf{LCA}}%
(A_{p},\mathbb{T})=A_{p}^{\vee}$. The map $e$ is continuous as $\mathbb{Q}%
_{p}/\mathbb{Z}_{p}$ carries the discrete topology. As a composition of
continuous maps, the output is indeed a continuous character. For the inverse
map, one needs to check that any character%
\[
\psi\colon A_{p}\longrightarrow\mathbb{T}%
\]
actually takes values in the subgroup $\mathbb{Q}_{p}/\mathbb{Z}_{p}%
\subset\mathbb{Q}/\mathbb{Z}$ (under the same maps as in Equation
\eqref{lwwaa7}). Every element in $A_{p}$ is topological $p$-torsion, so a
continuous character sends it to a topological $p$-torsion element in the
circle, so $\psi(A_{p})\subseteq\mathbb{T}_{p}$, and the latter group is just
the $p$-primary roots of unity by \cite[Lemma~2.6]{MR637201}, i.e. precisely
the image of $\mathbb{Q}_{p}/\mathbb{Z}_{p}$ under the map $e$ in Equation
\eqref{lwwaa7}. A further computation checks that the left $A_{p}$-module
structures are compatible under this isomorphism. The argument for
$p=\mathbb{R}$ is analogous, but more straightforward.
\end{proof}

\begin{proposition}
\label{prop_LCAAbHasKernelsAndCokernels}The category $\mathsf{LCA}_{A,ab}$ is
extension-closed in $\mathsf{LCA}_{A}$. As such, it is a fully exact subcategory.
Equipped with this exact structure, it is a split quasi-abelian category.
\end{proposition}

\begin{proof}
\textit{(Step 1) }First, we show that $\mathsf{LCA}_{A,ab}$ is
extension-closed. Suppose%
\[
X^{\prime}\hookrightarrow M\twoheadrightarrow X^{\prime\prime}%
\]
is a conflation in $\mathsf{LCA}_{A}$ with $X^{\prime},X^{\prime\prime}$
adelic blocks. By Corollary \ref{cor_Structure} we may write $M=K\oplus
X\oplus D$ and by the canonical filtration, applied to either map, we must
have $K=D=0$ (for example, $K$ must map to zero in $X^{\prime\prime}$ by the
filtration, so $K$ is a subobject of $X^{\prime}$, but compacts can only map
trivially to adelic blocks, so $K=0$). Like any extension-closed full
subcategory of an exact category, this equips $\mathsf{LCA}_{A,ab}$ with an
exact structure and then renders $\mathsf{LCA}_{A,ab}$ fully exact in
$\mathsf{LCA}_{A}$.\newline\textit{(Step 2) }We show that each $\mathsf{J}%
_{A}^{(S)}$ has kernels. This just uses that kernels (in the abelian category
of all modules) in each of the cases $\mathfrak{A}_{p},A_{p}$ or
$A_{\mathbb{R}}$ are finitely generated and since all these rings are
hereditary (\cite[Theorem~21.4]{MR1972204}), must again also be projective.
As all $\mathsf{J}_{A}^{(S)}$ have kernels, so has the $2$-colimit
$\mathsf{J}_{A}^{(\infty)}$. This also applies to $A^{op}$, so $\mathsf{J}%
_{A^{op}}^{(\infty)}$ also has all kernels. Under the equivalence of Proposition
\ref{prop_J_Duality} this implies that $\mathsf{J}_{A}^{(\infty)}$ has all
cokernels. Finally, invoke Proposition \ref{prop_IdentifyJ}. Having all
kernels and cokernels, we deduce that the category is quasi-abelian.\newline%
\textit{(Step 3)} It remains to show that $\mathsf{LCA}_{A,ab}$ is split
exact. We use Proposition \ref{prop_IdentifyJ} again. Each category
$\mathsf{J}_{A}^{(S)}$ is split exact by Lemma
\ref{lemma_ProductsOfSplitExact}. Any exact sequence in the $2$-colimit
$\mathsf{J}_{A}^{(\infty)}$ stems from $\mathsf{J}_{A}^{(S)}$ for $S$ big
enough; and induces a splitting in the $2$-colimit.
\end{proof}

\begin{example}\label{Example:AdelicsAreNotAbelian}
The category $\mathsf{LCA}_{A,ab}$ is not an abelian category. Suppose
$A=\mathbb{Q}$. There is an injective continuous morphism%
\[
r:\mathbb{A}\longrightarrow\mathbb{A}%
\]
given by multiplication with the ad\`{e}le $(2,3,5,\ldots,1)$, i.e. the
multiplication with $p$ on $\mathbb{Q}_{p}$, and the identity on $\mathbb{R}$.
That this is continuous follows for example from the ad\`{e}les being a
topological ring, or by using $\mathsf{J}_{A}^{(\infty)}$ instead and the
correspondence of Proposition \ref{prop_IdentifyJ}. If $\mathsf{LCA}_{A,ab}$ were an
abelian category, the morphism $r$ would have the analysis%
\[
r:\mathbb{A}\twoheadrightarrow\operatorname*{coim}(r)\hookrightarrow
\mathbb{A}\text{.}%
\]
However, the categorical coimage in $\mathsf{LCA}_{A}$ agrees with the
set-theoretic image, and thus the latter would have to be closed in
$\mathbb{A}$. As the category is split exact by Proposition
\ref{prop_LCAAbHasKernelsAndCokernels}, this would force $\mathbb{A}$ to be a
non-trivial direct summand of itself. While such a thing can happen in
categories, it cannot happen in $\mathsf{J}_{A}^{(\infty)}$ by a rank
consideration for the topological $p$-torsion parts.
\end{example}

\begin{theorem}
\label{thm_AdelicBlocksAreProjectiveAndInjective}Every adelic block $Y$ is an
injective and projective object in $\mathsf{LCA}_{A}$. No other objects are
simultaneously injective and projective.
\end{theorem}

\begin{proof}
We show that any $Y\in\mathsf{LCA}_{A,ab}$ is injective. To this end, it
suffices to show that $\operatorname*{Ext}^{1}(M,Y)=0$ holds for all
$M\in\mathsf{LCA}_{A}$. By Corollary \ref{cor_Structure} we have $M\simeq
K\oplus X\oplus D$ with $K$ a compact $A$-module, $X$ adelic and $D$ discrete.
We have $\operatorname*{Ext}^{1}(D,Y)=0$ since $D$ is projective by Lemma
\ref{lemma_CharacterizeInjectivesAndProjectives}, and $\operatorname*{Ext}%
^{1}(X,Y)=\operatorname*{Ext}_{\mathsf{LCA}_{A,ab}}^{1}(X,Y)=0$ since the subcategory of adelic blocks is split exact by Proposition
\ref{prop_LCAAbHasKernelsAndCokernels}. Thus, we only need to show that
$\operatorname*{Ext}^{1}(K,Y)=0$. Suppose $\phi\in\operatorname*{Ext}%
^{1}(K,Y)$ is arbitrary. It is represented by a conflation%
\begin{equation}
Y\overset{\alpha}{\rightarrowtail}Q\twoheadrightarrow K\text{.}\label{lwmx2}%
\end{equation}
By Lemma \ref{lemma_ExtensionFacts} (2), as $Y$ and $K$ are quasi-adelic, so
is $Q$, and then $Q\simeq X\oplus C$ with $X$ adelic (and possibly different
from the earlier use of $X$) and $C$ compact. We now show that the splitting
for $X$ in this direct sum produces a splitting of $\alpha$. To this end, we
form the pullback of the inflation with $C\rightarrowtail Q$, giving the
commutative diagram%
\[
\xymatrix{
Y \cap C \ar@{>->}[d] \ar@{>->}[r] & C \ar@{>->}[d] \ar@
{->>}[r] & C/(Y \cap C) \ar@{->}[d]^{\tau} \\
Y \ar@{>->}[r]_{\alpha } & Q \ar@{->>}[r] & K.
}
\]
The downward arrows are admissible (clear for left and middle, and the right
is between compact $A$-modules, which form an abelian category). As
$\mathsf{LCA}_{A}$ is quasi-abelian, we can apply the Snake Lemma in the style
of \cite[Corollary 8.13]{MR2606234}. We get an exact sequence%
\begin{equation}
\ker\tau\rightarrowtail Y/(Y\cap C)\overset{\widetilde{\alpha}}{\longrightarrow
}X\twoheadrightarrow\operatorname*{coker}\tau\text{,}\label{lwmx1}%
\end{equation}
where $\widetilde{\alpha}$ is induced from $\alpha$. However, we must have $Y\cap
C=0$ since this is a compact $A$-module, which cannot be non-zero inside an
adelic block since that would violate the Hoffmann--Spitzweck filtration
(Corollary \ref{cor_Structure}). As $Y$ is adelic and $\ker\tau$ compact, we
must have $\ker\tau=0$, and then $\operatorname*{coker}\tau\cong K/C$. Thus,
Equation \eqref{lwmx1} simplifies to the conflation%
\[
Y\overset{\widetilde{\alpha}}{\rightarrowtail}X\twoheadrightarrow K/C\text{,}%
\]
showing that $K/C$ is simultaneously a compact $A$-module (as $K$ and $C$
are), and adelic as $X$ and $Y$ are. Thus, $K/C=0$, showing that the map
$\widetilde{\alpha}:Y\rightarrowtail X$ was an isomorphism to start with, and $K\cong C$.
The diagram of the Snake Lemma thus yields the commutative diagram
\[
\xymatrix@R=0.035in{
&  Q \ar@{->>}[dd]^{\text{quot}} \\
Y \ar@{>->}[ur]^{\alpha } \ar[dr]_{\widetilde{\alpha }}^{\cong } \\
& X,
}
\]
showing that $\alpha$ (coming from our input conflation) is a splitting for the direct sum splitting
$Q\cong X\oplus C$.
It follows that $\phi=0$. Thus, $\operatorname*{Ext}%
^{1}(K,Y)=0$. It follows that $Y$ is injective. Finally, by Pontryagin double
dualization $Y\cong(Y^{\vee})^{\vee}$, and since $Y^{\vee}$ is an adelic block
in $\mathsf{LCA}_{A^{op}}$, and thus an injective object by the above part of
the proof, $(Y^{\vee})^{\vee}$ is projective. 

For the converse: Suppose a
discrete $A$-module $D$ is injective. For every element $d\in D$ there is the
map of discrete right $A$-modules $A\rightarrow D$, $a\mapsto da$. Lift this
map along the inflation in $A\rightarrowtail\mathbb{A}_{A}\twoheadrightarrow
\mathbb{A}_{A}/A$ using that $D$ is injective. But by the Hoffmann--Spitzweck
filtration, any morphism $\mathbb{A}_{A}\rightarrow D$ from the ad\`{e}les of
$A$ to $D$ must be zero, so $d=0$ by commutativity. Hence, $D=0$. By
Pontryagin duality, no compact module can be projective. Corollary
\ref{cor_Structure} now implies the claim.
\end{proof}


\section{One-sided exact categories and quotients}\label{Section:OneSided}

Our approach to the calculation of the $K$-theory spectrum is similar to the one in \cite{paper1}, that is, we construct two subsequent quotients of the category $\mathsf{LCA}_A$.  First, we take the quotient by the category $\mathsf{LCA}_{A,com}$ of compact $A$-modules; secondly, we take a further quotient by the category $\mathsf{LCA}_{A,ab}$ of adelic blocks.  These quotients are taken using the framework in \cite{hr2, hr}, that means we show that they are inflation-percolating subcategories (see Definition \ref{definition:InflationPercolating}).

To better understand the first quotient, the following property will be useful: we say that a (one-sided) exact category has \emph{admissible cokernels} (see Definition \ref{Definition:AdmissibleCokernels}) if every morphism has a cokernel and this cokernel is admissible, that is, the cokernel is a deflation.  We show in this section that the property of having admissible cokernels is stable under both quotients of inflation-exact categories and taking exact hulls of such categories. These observations allow us to bypass some technical difficulties in the computation of the $K$-theory spectra of $\mathsf{LCA}_A$ in the next section.

\subsection{Quotients of (one-sided) exact categories by percolating subcategories}

A \emph{conflation category} $\EE$ is an additive category $\EE$ together with a chosen class of kernel-cokernel pairs (closed under isomorphisms), called \emph{conflations}. We refer to the kernel-part of a conflation as an \emph{inflation} (depicted by $\inflation$) and the cokernel-part as a \emph{deflation} (depicted by $\deflation$).  An additive functor $F\colon \CC\to \DD$ of conflation categories is called \emph{exact} (or \emph{conflation-exact}) if it maps conflations to conflations. 
 
\begin{definition}\label{Definition:InflationExact}
 A conflation category $\EE$ is called an \emph{inflation-exact category} if the following axioms are satisfied:
	\begin{enumerate}[label=\textbf{L\arabic*},start=0]
		\item\label{L0} For each $X\in \EE$, the map $0\to X$ is an inflation.
		\item\label{L1} Inflations are closed under composition.
		\item\label{L2} Pushouts along inflations exist and inflations are stable under pushouts.
	\end{enumerate}
	The notion of a \emph{deflation-exact category} is defined dually, the dual axioms are called \textbf{R0}-\textbf{R2}.
\end{definition}

In addition to the axioms listed above, we formulate a (weakly idempotent complete) version of Quillen's obscure axiom.

\begin{definition}
	Let $\EE$ be a conflation category. We define the following axiom:
	\begin{enumerate}
		\myitem{\textbf{L3}$^{+}$}\label{L3+} If $f\colon X\to Y$ and $g\colon Y\to Z$ are maps such that $g\circ f$ is an inflation, then $f$ is an inflation.
	\end{enumerate}
	The dual axiom is called \textbf{R3}$^+$.
\end{definition}

\begin{remark}
\makeatletter%
\hyper@anchor{\@currentHref}%
\makeatother\label{Remark:SnakeLemma}%
	\begin{enumerate}
	  \item What we call axiom \ref{L0} is stronger than the corresponding axiom in \cite{BazzoniCrivei13}.  The axiom as presented here is necessary for all split kernel-cokernel pairs to be conflations.
		\item A Quillen exact category is a two-sided exact category in the sense that it is a conflation category that is both inflation-exact and deflation-exact (see \cite[Appendix~A]{Keller90}).
		\item An inflation-exact category satisfies axiom \ref{L3+} if and only if the Snake Lemma holds (see \cite[Theorem~1.2]{HenrardvanRoosmalen20Obscure}).
	\end{enumerate}
\end{remark}

The theory of one-sided exact categories parallels the theory of exact categories and many notions from exact categories can be transferred easily to the one-sided exact setting.  For example, the notions of \emph{admissible morphisms} or \emph{strict morphisms} (morphisms that admit a deflation-inflation factorization) carries over verbatim from \cite{MR2606234}.  Similarly, one can define \emph{acyclic} or \emph{exact} sequences.  Given an inflation-exact category $\EE$, the bounded derived category $\Db(\EE)$ is defined as the Verdier localization $\Kb(\EE)/\Acb(\EE)$ where $\Kb(\EE)$ is the bounded homotopy category of cochain complexes and $\Acb(\EE)$ is the triangulated subcategory (not necessarily closed under isomorphisms) of $\Kb(\EE)$ consisting of bounded acyclic sequences (see \cite[Section 7]{BazzoniCrivei13}). In a similar vein, the stable $\infty$-category $\mathsf{D}_{\infty}^{\mathsf{b}}(\EE)$ is defined as well (see \cite[Section 8]{hr2}).

Let $\EE$ be an inflation-exact category. The \emph{exact hull} $\EE^{\mathsf{ex}}$ of $\EE$ is the extension-closure of $i(\EE)\subseteq \Db(\EE)$ where $i\colon \EE\to \Db(\EE)$ is the natural embedding (see \cite{hr2} or \cite[Proposition~I.7.5]{Rosenberg11}). Following \cite{hr2}, the exact hull is endowed with the structure of an exact category via the triangle structure of $\Db(\EE)$. Moreover, both the $2$-natural embedding $j\colon\EE\to\EE^{\mathsf{ex}}$ and the embedding $\EE^{\mathsf{ex}}\to \Db(\EE)$ lift to triangle equivalences $\Db(\EE)\simeq \Db(\EE^{\mathsf{ex}})$. In fact, $j$ lifts to an equivalence $\mathsf{D}_{\infty}^{\mathsf{b}}(\EE)\simeq \mathsf{D}_{\infty}^{\mathsf{b}}(\EE^{\mathsf{ex}})$ on the $\infty$-derived categories (these are stable $\infty$-categories in the sense of \cite{LurieHA}, see \cite{hr2} and the references therein).

We will use the following lemma (see \cite[Lemma~2.26]{HenrardKvammevanRoosmalenWegner21}).

\begin{lemma}\label{Lemma:Epi/MonoInHull}
	Let $\EE$ be an inflation-exact category. Let $f\colon X\to Y$ be a map in $\EE$. The map $f$ is an epimorphism (monomorphism) in $\EE$ if and only if $j(f)$ is an epimorphism (monomorphism) in $\EE^{\mathsf{ex}}$.
\end{lemma}

We recall the following two definitions:

\begin{definition}\label{definition:InflationPercolating}
	Let $\EE$ be a conflation category. A full additive subcategory $\AA\subseteq \EE$ is called an \emph{inflation-percolating subcategory} if the following axioms are satisfied:
	\begin{enumerate}[label=\textbf{P\arabic*},start=1]
		\item\label{P1} $\AA$ is a \emph{Serre subcategory}, meaning:
		\[\mbox{ If } A'\inflation A \deflation A'' \mbox{ is a conflation in $\EE$, then } A\in \AA \mbox{ if and only if } A',A''\in \AA.\]
		\item\label{P2} Any morphism $f\colon A\to X$ with $A\in \AA$ factors as $A\xrightarrow{g}A'\stackrel{h}{\inflation}X$ where $A'\in \AA$.
		\item\label{P3} For any composition $T\xrightarrow{t}Y\stackrel{p}{\deflation}Z$ which factors through some object $B\in \AA$, there exists a commutative diagram
		\[\xymatrix{
			T\ar@/^/[rrd]^t\ar@/_/[ddr]\ar@{.>}[rd] &&\\
			& P\ar@{>->}[r]^{i'}\ar@{->>}[d]^{p'} & Y\ar@{->>}[d]^{p}\\
			& A\ar@{>->}[r]^{i} & Z
		}\] with $A\in \AA$ and such that the square $PYAZ$ is a pullback square.
		
		\item\label{P4} For all maps $X\stackrel{f}{\to}Y$ that factor through $\AA$ and for any deflation $Y\stackrel{p}{\deflation}A$ with $A\in \AA$ such that $p\circ f=0$, the induced map $X\to \ker(p)$ factors through $\AA$.
	\end{enumerate}
	
	Similarly, $\AA\subseteq \EE$ is called a \emph{strictly inflation-percolating} (or \emph{admissibly inflation-percolating}) if the following three axioms are satisfied:
	\begin{enumerate}[label=\textbf{A\arabic*},start=1]
		\item\label{A1} $\AA$ is a \emph{Serre subcategory}.
		\item\label{A2} Any morphism $f\colon A\to X$ with $A\in \AA$ is admissible (with image in $\AA$), i.e.~factors as $A\deflation A'\inflation X$ with $A'\in \AA$.
		\item\label{A3} Any cospan $\xymatrix{A\ar@{>->}[r]&Z&Y\ar@{->>}[l]}$ with $A\in \AA$ can be completed to a pullback square of the form:
		\[\xymatrix{P\ar@{>->}[r]\ar@{->>}[d] & Y\ar@{->>}[d]\\
		A\ar@{>->}[r] & Z}\]
	\end{enumerate}
\end{definition}

\begin{remark}%
\makeatletter%
\hyper@anchor{\@currentHref}%
\makeatother%
\label{Remark:StrictPercolatingIsAbelian}
	\begin{enumerate}
		\item Axiom \ref{P3} does not assume any relation between the objects $A$ and $B$.
		\item In many applications, a percolating subcategory satisfies a slightly stronger version of axiom \ref{P2}.  We say that $\AA\subseteq \EE$ satisfies \emph{strong axiom \ref{P2}} if $g$ can be chosen as an epimorphism in axiom \ref{P2}.
		\item	Note that axiom \ref{A2} implies strong axiom \ref{P2}.
		\item A strictly inflation-percolating subcategory of an inflation-exact category is automatically an abelian subcategory (see \cite[Proposition~6.4]{hr}). 
		\item If $\EE$ is an exact category and $\AA\subseteq \EE$ is a full additive subcategory, then $\AA$ automatically satisfies axiom \ref{A3} (see \cite[Proposition~2.15]{MR2606234}).
	\end{enumerate}
\end{remark}

\begin{definition}
	Let $\EE$ be a conflation category and let $\AA\subseteq \EE$ be a full additive subcategory. An \emph{$\AA$-inflation} is an inflation with cokernel in $\AA$, similarly, an \emph{$\AA$-deflation} is a deflation with kernel in $\AA$. A \emph{weak $\AA$-isomorphism} is a finite composition of $\AA$-inflations and $\AA$-deflations. The set of weak isomorphisms is denoted by $S_{\AA}$.
\end{definition}

The following theorem summarizes the main results of \cite{hr2,hr} in a convenient form.

\begin{theorem}\label{Theorem:MainLocalizationTheorem}
	Let $\EE$ be an inflation-exact category and let $\AA\subseteq \EE$ be an inflation-percolating subcategory. Write $Q\colon \EE\to \EE[S_{\AA}^{-1}]$ for the localization functor with respect to the set of weak isomorphisms $S_{\AA}$. The following hold:
	\begin{enumerate}
		\item The set $S_{\AA}$ is a left multiplicative system.
		\item The smallest conflation structure on $\EE[S_{\AA}^{-1}]$ for which the functor $Q\colon \EE\to \EE[S_{\AA}^{-1}]$ is conflation-exact, is an inflation-exact structure.
		\item The functor $Q$ satisfies the $2$-universal property of a quotient in the category of inflation-exact categories.  This motivates the notation $\EE/\AA \coloneqq \EE[S_{\AA}^{-1}]$.
		\item The localization sequence $\AA\to\EE\xrightarrow{Q}\EE/\AA$ lifts to a Verdier localization sequence 
		\[\DAb(\EE)\to \Db(\EE)\to \Db(\EE/\AA).\]
		Here, $\DAb(\EE)$ is the thick triangulated subcategory of $\Db(\EE)$ generated by $\AA$.
		\item If the natural functor $\Db(\AA)\to \DAb(\EE)$ is an equivalence, we obtain an exact sequence 
		\[\mathsf{D}_{\infty}^{\mathsf{b}}(\AA)\to \mathsf{D}_{\infty}^{\mathsf{b}}(\EE)\to \mathsf{D}_{\infty}^{\mathsf{b}}(\EE/\AA)\]
		in the sense of \cite{MR3070515}.
		\item For any morphism $f$ in $\EE$, we have that $Q(f)=0$ if and only if $f$ factors through $\AA$.
	\end{enumerate}
	
	Moreover, if $\BB\subseteq \EE$ is a strictly inflation-percolating subcategory, then $\BB\subseteq \EE$ is inflation-percolating (in particular all of the above holds). In addition, the following properties hold:
	\begin{enumerate}
		\item Every map $f\in S_{\BB}$ is strict (this explains the terminology).
		\item The set $S_{\BB}$ is saturated, i.e.~ if $Q(f)$ is an isomorphism, then $f\in S_{\BB}$.
	\end{enumerate}
\end{theorem}

\subsection{Inflation-exact categories with admissible cokernels}

We will use the above results for an inflation-exact category with admissible cokernels (see Definition \ref{Definition:AdmissibleCokernels}).  We use this structure to avoid later technicalities.

\begin{definition}\label{Definition:AdmissibleCokernels}
	Let $\EE$ be an inflation-exact category.  We say that $\EE$ has \emph{admissible cokernels} if every morphism has a cokernel which is a deflation.
\end{definition}

The following description is \cite[Proposition~4.9]{HenrardKvammevanRoosmalenWegner21}.

\begin{proposition}
The following are equivalent for an inflation-exact category $\EE$:
\begin{enumerate}
	\item $\EE$ has admissible kernels,
	\item every map in $\EE$ has an epi-inflation factorization, that is, every map $f$ factors as $f = i \circ k$ where $k$ is a kernel and $i$ is an inflation.
\end{enumerate}
\end{proposition}

\begin{remark}\label{remark:AdditiveCoregular}
Following \cite{BarrGrilletVanOsdol71,BorceuxBourn04}, a category is called \emph{coregular} if (1) every morphism $f$ has factors as $f = k \circ p$ where $p$ is an epimorphism and $k$ is a kernel map and (2) kernels are stable under pushouts.  It is shown in \cite[Proposition~4.11]{HenrardKvammevanRoosmalenWegner21} that an additive coregular category is an inflation-exact category with admissible cokernels; here, the conflations are given by all kernel-cokernel pairs.
\end{remark}

\begin{remark}
	\begin{enumerate}
		\item One readily verifies that an inflation-exact category $\EE$ having admissible cokernels satisfies axiom \ref{L3+}.  In particular, the Snake Lemma holds in $\EE$ by Remark \ref{Remark:SnakeLemma}. 
		\item	Every quasi-abelian category $\EE$ is inflation-exact and has admissible cokernels.  In fact, a conflation category $\EE$ is a quasi-abelian category if and only if it is both an inflation-exact category having admissible cokernels and a deflation-exact category having admissible kernels.
	\end{enumerate}
\end{remark}

Having admissible cokernels is preserved under localizations at inflation-percolating subcategories as well as under taking the exact hull (see the duals of \cite[Theorem~5.9 and Proposition~6.5]{HenrardKvammevanRoosmalenWegner21}).

\begin{proposition}\label{Proposition:AdmissibleCokernelsAreStableUnderQuotients}
	Let $\EE$ be an inflation-exact category and let $\AA\subseteq \EE$ be an inflation-percolating subcategory. If $\EE$ has admissible cokernels, then so does $\EE/\AA$.
\end{proposition}

\begin{proposition}\label{Proposition:ACPropertyLiftsToHull}
	Let $\EE$ be an inflation-exact category. If $\EE$ has admissible cokernels, then so does $\EE^{\mathsf{ex}}$.
\end{proposition}


The following is an analogue of \cite[Proposition 2.6]{paper1}.

\begin{proposition}\label{Proposition:StrongP2LiftsToHull}
	Let $\EE$ be an inflation-exact category and let $\AA\subseteq \EE$ be a full additive subcategory satisfying strong axiom \ref{P2}. If $\EE$ has admissible cokernels, then $\AA\subseteq \EE^{\mathsf{ex}}$ satisfies strong axiom \ref{P2}.
\end{proposition}

\begin{proof}
	Let $f\colon A\to X$ be a map in $\EE^{\mathsf{ex}}$ with $A\in \AA$. By \cite[Corollary 7.5]{hr2}, there is a conflation $X\stackrel{i}{\inflation}Y\stackrel{p}{\deflation}Z$ in $\EE^{\mathsf{ex}}$ such that $Y,Z\in \EE$. By strong axiom \ref{P2}, the composition $i\circ f$ factors as $A\stackrel{e}{\rightarrow}A'\inflation Y$ in $\EE$ with $A'\in \AA$ and $e$ epic. By Lemma \ref{Lemma:Epi/MonoInHull}, the latter yields an epi-inflation factorization of $i\circ f$ in $\EE^{\mathsf{ex}}$.
	
	By Proposition \ref{Proposition:ACPropertyLiftsToHull}, $f$ admits an epi-inflation factorization $A\stackrel{e'}{\rightarrow}T\stackrel{k}{\inflation}X$ in $\EE^{\mathsf{ex}}$. By axiom \ref{L1}, the composition $i\circ k$ is an inflation in $\EE^{\mathsf{ex}}$. Hence we obtain the epi-inflation factorization $A\stackrel{e'}{\rightarrow}T\stackrel{i\circ k}{\inflation} X$ of $i\circ f$. As epi-inflation factorizations are unique up to isomorphism, we conclude that $T\cong A'\in \AA$. This shows strong axiom \ref{P2}.
\end{proof}

\section{\texorpdfstring{$K$}{K}-theory computations}\label{KTheoryComputations}

We rely on the nomenclature of \emph{localizing invariants} from
\cite[Definition 8.1]{MR3070515}. Following their notation,
$\operatorname*{Cat}_{\infty}^{\operatorname*{ex}}$ is the $\infty$-category
of small stable $\infty$-categories. An invariant is \emph{weakly localizing}
if one drops the requirement to commute with filtering colimits.

\begin{theorem}
\label{thm_KAdelicBlocks}Suppose $\mathbb{K}\colon\operatorname*{Cat}_{\infty }^{\operatorname*{ex}}\longrightarrow\mathsf{A}$ is a localizing invariant with values in a stable presentable $\infty$-category $\mathsf{A}$. If $\mathbb{K}$ commutes with countable products, then there is a natural equivalence%
\[
\mathbb{K}(\mathsf{LCA}_{A,ab})\overset{\sim}{\longrightarrow}\left.
\underset{S}{\operatorname*{hocolim}}\left(  \prod_{p\notin S}\mathbb{K}%
(\mathfrak{A}_{p})\times\prod_{p\in S}\mathbb{K}(A_{p})\right)  \right.
\text{,}%
\]
where $S$ runs through all finite subsets of $\Pl$, partially ordered by inclusion. If $A$ is commutative,
all categories involved are naturally categories with dualities (as detailed
in \S \ref{subsect_Duality}) and if $\mathbb{K}$ is an invariant incorporating
duality, the equivalence still holds.
\end{theorem}

\begin{proof}
This follows directly from the equivalence of exact categories in Proposition
\ref{prop_IdentifyJ}, which implies the equivalence of the attached stable
$\infty$-categories of bounded complexes (modulo acyclic complexes). Regarding
duality, combine this with Proposition \ref{prop_J_Duality}.
\end{proof}

Non-connective $K$-theory is an example of a localizing invariant commuting
with infinite products by \cite{kwinfproducts}. Since all our input categories
stem from exact categories, one could get by with the classical result of
\cite{MR1351941}.

Weakly localizing invariants decidedly do not suffice for the above because of
the colimit in Definition \ref{def_A} (so TC or negative cyclic homology
cannot be taken for $\mathbb{K}$).

\begin{theorem}
\label{thm_FiberSeq}Suppose $\mathbb{K}\colon\operatorname*{Cat}_{\infty
}^{\operatorname*{ex}}\longrightarrow\mathsf{A}$ is a weakly localizing
invariant with values in a stable presentable $\infty$-category $\mathsf{A}$.
There is a natural fiber sequence in $\mathsf{A}$,%
\[
\mathbb{K}(A)\longrightarrow\mathbb{K}(\mathsf{LCA}_{A,ab})\longrightarrow
\mathbb{K}(\mathsf{LCA}_{A})\text{.}%
\]
\end{theorem}

This result does not require $\mathbb{K}$ to commute with countable products.\\


The proof of Theorem \ref{thm_FiberSeq} follows Section 4 of \cite{paper1} closely.

The following is the analogue of \cite[Proposition~4.1]{paper1}.  Recall that $\mathsf{LCA}_{A,com}$ is the full subcategory of $\mathsf{LCA}_{A}$ consisting of all compact $A$-modules.

\begin{proposition}\label{Proposition:KillingCompacts}
The category $\mathsf{LCA}_{A,com}$ is a strictly inflation-percolating subcategory of $\mathsf{LCA}_A$ and the localization functor 
\[Q_c\colon \mathsf{LCA}_A\to \mathsf{LCA}_A/\mathsf{LCA}_{A,com}(\eqqcolon \EE)\]
induces an exact sequence of stable $\infty$-categories
	\[\mathsf{D}_{\infty}^{\mathsf{b}}(\mathsf{LCA}_{A,com})\to \mathsf{D}_{\infty}^{\mathsf{b}}(\mathsf{LCA}_{A})\to \mathsf{D}_{\infty}^{\mathsf{b}}(\mathcal{E}).\]
	As the category $\mathsf{LCA}_{A,com}$ admits all coproducts, there is an equivalence $\mathbb{K}(\mathsf{D}_{\infty}^{\mathsf{b}}(\mathsf{LCA}_{A}))\simeq \mathbb{K}(\mathcal{E})$.
	
	Moreover, as $\mathsf{LCA}_{A}$ is a quasi-abelian category, it has admissible cokernels, and thus, $\EE$ has admissible cokernels as well.
\end{proposition}

\begin{proof}
	That $\mathsf{LCA}_{A,com}\subseteq \mathsf{LCA}_A$ is a strictly inflation-percolating subcategory is shown in \cite[Proposition~8.34]{hr}. As $\mathsf{LCA}_{A,com}$ contains enough injectives for $\mathsf{LCA}_A$, $\Db(\mathsf{LCA}_{A,com})\to \mathsf{D}^{\mathsf{b}}_{\mathsf{LCA}_{A,com}}(\mathsf{LCA}_{A})$ lifts to a triangle equivalence. The remainder of the statement follows from Theorem \ref{Theorem:MainLocalizationTheorem}. The last statement follows from Proposition \ref{Proposition:AdmissibleCokernelsAreStableUnderQuotients}.
\end{proof}

In light of Corollary \ref{cor_Structure}, we now wish to further annihilate the adelic blocks from $\mathcal{E}$. To that end, write $\mathcal{V}$ for the full additive subcategory of $\mathcal{E}$ generated by the adelic blocks. Unfortunately, it is not clear whether $\mathcal{V}\subseteq \mathcal{E}$ is an inflation-percolating subcategory. However, as $\mathbb{K}(\EE^{\mathsf{ex}})\simeq \mathbb{K}(\EE)$, for our purposes, we content ourselves with showing that $\mathcal{V}$ is an inflation-percolating subcategory of the exact hull $\mathcal{E}^{\mathsf{ex}}$. We first consider the following analogue of \cite[Lemma~4.2]{paper1}.

\begin{lemma}%
\makeatletter%
\hyper@anchor{\@currentHref}%
\makeatother\label{Lemma:Q_cInducesNaturalEquivalence}%
	\begin{enumerate}
		\item Let $V$ be an adelic block, the localization functor $Q_c$ induces a natural equivalence 
		\[Q_c\colon \Hom_{\mathsf{LCA}_A}(-,V)\to \Hom_{\mathcal{E}}(Q_c(-),Q_c(V)).\]
			In particular, it follows that $V$ is injective in $\mathcal{E}$.
		\item The category $\mathcal{V}$ is equivalent to the category $\mathsf{LCA}_{A,ab}$.
	\end{enumerate}
\end{lemma}

\begin{proof}
	\begin{enumerate}
		\item Let $f\in \Hom_{\mathsf{LCA}_A}(X,V)$ be a map such that $Q_c(f)=0$. By the quotient/localization theory of percolating subcategories, $\mathcal{E}$ is the localization $\mathsf{LCA}_{A}[S_{\mathsf{LCA}_{A,com}}^{-1}]$. Here, $S_{\mathsf{LCA}_{A,com}}$ is the collection of admissible morphisms with kernel and cokernel belonging to $\mathsf{LCA}_{A,com}$ and $S_{\mathsf{LCA}_{A,com}}$ is a left multiplicative system. As such, $Q_c(f)=0$ implies that there exists a $t\colon V\to Y$ in $S_{\mathsf{LCA}_{A,com}}$ such that $t\circ f=0$. It follows that $f$ factors through $\ker(t)$. On the other hand, $\ker(t)\rightarrowtail V$ must be zero by Lemma \ref{lemma_NoMapsFromQAToDiscrete} as $\ker(t)\in \mathsf{LCA}_{A,com}$. This shows that $f=0$ and thus $Q_c$ is injective.
		
		To show that $Q_c$ is surjective, let $g\in \Hom_{\mathcal{E}}(X,V)$ be represented by the roof $X\stackrel{f}{\rightarrow} Y \stackrel{s}{\leftarrow} V$ with $s\in S_{\mathsf{LCA}_{A,com}}$. By definition $s$ is admissible and $\ker(s)\in \mathsf{LCA}_{A,com}$. As $V$ has no non-trivial compact submodules, $s$ is an inflation. As $V$ is injective by Theorem \ref{thm_AdelicBlocksAreProjectiveAndInjective}, $s$ splits. Let $t\colon Y\to V$ be a splitting, i.e.~$t\circ s=1_V$, then $Q_c(t\circ f)=g$. This shows the desired bijection.
		
		As $V$ is injective in $\mathsf{LCA}_{A}$, $\Hom_{\mathsf{LCA}_A}(-,V)$ is an exact functor. Hence $\Hom_{\EE}(-,V)$ is an exact functor and thus $V$ is injective in $\mathcal{E}$.
		\item This follows immediately from the above natural equivalence. \qedhere
	\end{enumerate}
\end{proof}

The main difference with \cite{paper1} is that $\VV$ cannot be a strictly inflation-percolating subcategory of $\EE$. Indeed, by Example \ref{Example:AdelicsAreNotAbelian}, $\mathsf{LCA}_{A,ab}$ is not abelian and thus Remark \ref{Remark:StrictPercolatingIsAbelian} implies that $\VV$ cannot be strictly percolating. Nonetheless, we obtain the following analogue of \cite[Proposition~4.3]{paper1}.

\begin{proposition}
	The subcategory $\VV\subseteq \EE^{\mathsf{ex}}$ is an inflation-percolating subcategory.
\end{proposition}

\begin{proof}
	We first show strong axiom \ref{P2}. By Proposition \ref{Proposition:ACPropertyLiftsToHull}, the category $\EE^{\mathsf{ex}}$ has admissible cokernels. By Proposition \ref{Proposition:StrongP2LiftsToHull}, it now suffices to show that $\VV\subseteq \EE$ satisfies strong axiom \ref{P2}.	To that end, let $f\colon V\to X$ be a map in $\EE$ with $V\in \VV$. As $\EE$ has admissible cokernels, $f$ admits an epi-inflation factorization $V \stackrel{g}{\to} T\stackrel{h}{\inflation}X$ in $\EE$. As $\EE=\mathsf{LCA}_{A}[S_{\mathsf{LCA}_{A,com}}^{-1}]$ is a localization and $S_{\mathsf{LCA}_{A,com}}$ is a left multiplicative system, the map $g$ can be represented as a roof $V\stackrel{g'}{\to} T' \stackrel{s}{\leftarrow} T$ with $s\in S_{\mathsf{LCA}_{A,com}}$. Note that $g'$ is an epimorphism in $\EE$ and thus its cokernel is zero in $\EE$. As the localization functor $Q_c$ commutes with cokernels, the cokernel of $g'$ belongs to $\mathsf{LCA}_{A,com}$. Hence, as $\mathsf{LCA}_{A}$ has admissible cokernels, we obtain the sequence in $\mathsf{LCA}_{A}$:
	\[V \to \im(g') \inflation T'\deflation C.\]
	Here, $C\in \mathsf{LCA}_{A,com}$, $V\to \im(g')$ is epic and the composition $V\to \im(g')\inflation T'$ is $g'$. It follows from Corollary \ref{cor_Structure} that epimorphic quotient of adelic blocks are quasi-adelic blocks, and thus $\im(g')\in \mathsf{LCA}_{A,qab}$. By Lemma \ref{lemma_DecompQAB}, $\im(g')$ is isomorphic to an adelic block in $\EE$ as compacts are annihilated. It follows that $\im(g')\in \VV$. This shows that $\VV\subseteq \EE$ satisfies strong axiom \ref{P2} as $Q_c$ maps epics to epics.
	
	We now claim that every object of $\VV$ is injective in $\EE^{\mathsf{ex}}$. Let $V\in \VV$. By Lemma \ref{Lemma:Q_cInducesNaturalEquivalence}, $V$ is injective in $\EE$. As $\EE^{\mathsf{ex}}$ is simply the extension-closure of $\EE$ in $\EE^{\mathsf{ex}}$, one readily verifies that $V$ remains injective in $\EE^{\mathsf{ex}}$. It follows that $\VV\subseteq \EE^{\mathsf{ex}}$ is extension-closed.
	
	Now let $X\stackrel{i}{\inflation} V \stackrel{p}{\deflation} Z$ be a conflation in $\EE^{\mathsf{ex}}$ with $V\in \VV$. By strong axiom \ref{P2}, $Z\cong \im(p)\in \VV$. It follows from Proposition \ref{prop_LCAAbHasKernelsAndCokernels} that $X\in \VV$. This shows axiom \ref{P1}.
	
	As $\EE^{\mathsf{ex}}$ is exact, axiom \ref{P3} is automatic. Axiom \ref{P4} follows from (the dual of) \cite[Proposition 4.11]{hr} and Proposition \ref{prop_LCAAbHasKernelsAndCokernels}. This concludes the proof.
\end{proof}

\begin{corollary}
	The localization functor $Q_{\VV}\colon \EE^{\mathsf{ex}}\to \FF(=\EE^{\mathsf{ex}}/\VV)$ induces a fiber sequence
	\[\mathbb{K}(\VV)\to \mathbb{K}(\EE^{\mathsf{ex}})\to \mathbb{K}(\FF).\]
\end{corollary}

\begin{proof}
	By Lemma \ref{Lemma:Q_cInducesNaturalEquivalence}, $\VV\subseteq \EE^{\mathsf{ex}}$ contains enough injectives. The result now follows from Theorem \ref{Theorem:MainLocalizationTheorem}.
\end{proof}

\begin{proposition}
	The composition functor $\mathsf{LCA}_A\to \EE\to \EE^{\mathsf{ex}}\to \FF$ is $2$-universal with respect to conflation-exact functors $F\colon \mathsf{LCA}_A\to \CC$ with $\CC$ exact such that $F(\mathsf{LCA}_{A,ab})=0$. 
\end{proposition}

\begin{proof}
	This follows by simply glueing together all $2$-universal properties.
\end{proof}

Our next goal is to show that $\FF$ is equivalent to the category $\Mod(A)/\smod(A)$ (see \cite[Proposition 4.9]{paper1}). Consider the localization functor  $Q_A\colon \Mod(A)\to \Mod(A)/\smod(A)$. We write $D\colon\mathsf{LCA}_A\to \mathsf{LCA}_{A,dis}\simeq \Mod(A)$ for the functor mapping an object to its torsion-free part with respect to the torsion theory in Proposition \ref{prop_Structure}. Note that $D$ is not conflation-exact (the standard ad\`{e}le sequence $A\inflation \mathbb{A}\deflation \mathbb{A}/A$ is mapped to $A\to 0\to 0$). However, we need not to change much to remedy this (see also \cite[Proposition 4.7]{paper1}).

\begin{proposition}
	The functor $Q_A\circ D\colon \mathsf{LCA}_A\to \Mod(A)/\smod(A)$ is conflation-exact.
\end{proposition}

\begin{proof}
	Let $X\inflation Y\deflation Z$ be a conflation in $\mathsf{LCA}_A$. By Proposition \ref{prop_Structure}, we obtain the following commutative diagram:
	\[\xymatrix{                                                                                                    
		W_X\ar@{>->}[r]\ar@{>.>}[d] & X\ar@{->>}[r]\ar@{>->}[d] & D_X\ar[d]^g\\
		W_Y\ar@{>->}[r] & Y\ar@{->>}[r] & D_Y
	}\] Note that the left downwards arrow is an inflation by \cite[Proposition 7.6]{MR2606234}. Clearly the map $g$ is strict. Applying the Short Snake Lemma (\cite[Corollary 8.13]{MR2606234}), we obtain an exact sequence
	\[\ker(g) \stackrel{i}{\inflation} W_Y/W_X \to Z\deflation \coker(g).\]
	By Lemma \ref{lemma_ExtensionFacts}, $W_Y/W_X\in \mathsf{LCA}_{A,qab}$ and hence $\coker(i)\in \mathsf{LCA}_{A,qab}$. Note that the conflation $\coker(i)\inflation Z \deflation \coker(g)$ thus is the torsion sequence of $Z$ (from Proposition \ref{prop_Structure}). Thus $\coker(g)\cong D_Z$.
	
	Lastly, it follows from Lemma \ref{lemma_ExtensionFacts} that $\ker(g)$ is finitely generated and discrete. Hence we find a conflation $Q_A(D_X)\inflation Q_A(D_Y)\deflation Q_A(D_Z)$ as required.
\end{proof}

\begin{corollary}
	The functor $Q_A \circ D\colon \mathsf{LCA}_A\to  \Mod(A)/\smod(A)$ induces an equivalence $\FF\to \Mod(A)/\smod(A)$ of abelian categories. 
\end{corollary} 

\begin{proof}
	The proof is identical to \cite[Construction~4.8 and Proposition 4.9]{paper1}.
\end{proof}

We are now in a position to prove Theorem \ref{thm_FiberSeq}.

\begin{proof}[Proof of Theorem \ref{thm_FiberSeq}]
Consider the following natural commutative diagram
\[\xymatrix{
	\smod(A)\ar[r]\ar[d] & \Mod(A)\ar[r]\ar[d] & \Mod(A)/\smod(A)\ar[d]^{\simeq}\\
	\VV\ar[r] & \EE^{\mathsf{ex}}\ar[r] & \FF
}\] whose rows are localization sequences. Note that this diagram lifts to a diagram on the bounded derived $\infty$-categories such that the rows are exact. As $\VV\simeq \mathsf{LCA}_{A,ab}$ by Lemma \ref{Lemma:Q_cInducesNaturalEquivalence}, we obtain the following bicartesian square of stable $\infty$-categories:
\[\xymatrix{
	\mathsf{D}_{\infty}^{\mathsf{b}}(\smod(A))\ar[r]\ar[d] & \mathsf{D}_{\infty}^{\mathsf{b}}(\Mod(A))\ar[d]\\
	\mathsf{D}_{\infty}^{\mathsf{b}}(\mathsf{LCA}_{A,ab})\ar[r] & \mathsf{D}_{\infty}^{\mathsf{b}}(\EE^{\mathsf{ex}})
}\] Again, using the Eilenberg swindle, every object of $\mathsf{D}_{\infty}^{\mathsf{b}}(\Mod(A))$ gets trivialized under a localizing invariant $\mathbb{K}\colon \mathsf{Cat}_{\infty}^{\mathsf{Ex}}\to \mathsf{A}$. Hence for each such $\mathbb{K}$, there is a fiber sequence $\mathbb{K}(\smod(A))\to \mathbb{K}(\mathsf{LCA}_{A,ab})\to \mathbb{K}(\EE^{\mathsf{ex}})$. By Proposition \ref{Proposition:KillingCompacts}, $\mathsf{D}_{\infty}^{\mathsf{b}}(\EE^{\mathsf{ex}})\simeq \mathsf{D}_{\infty}^{\mathsf{b}}(\mathsf{LCA}_A)$. The result follows.   
\end{proof}


All of the previous results give a complete computation of the $K$-theory
spectrum of $\mathsf{LCA}_{A}$. In the special case of the first homotopy
group $\pi_{1}K$ we recover a lot of well-known classical invariants:

\begin{theorem}
\label{thm_ComputeK1}Let $A$ be a finite-dimensional semisimple $\mathbb{Q}%
$-algebra and let $K$ denote Quillen's algebraic $K$-theory.

\begin{enumerate}
\item There is a natural isomorphism%
\[
K_{1}(\mathsf{LCA}_{A,ab})\overset{\sim}{\longrightarrow}\frac{J(A)}{J^{1}%
(A)}\text{,}%
\]
where $J(A)$ is Fr\"{o}hlich's id\`{e}le class group and $J^{1}$ the subgroup
of reduced norm one elements.

\item There is a natural isomorphism%
\[
K_{1}(\mathsf{LCA}_{A})\overset{\sim}{\longrightarrow}\frac{J(A)}%
{J^{1}(A)\cdot\operatorname*{im}A^{\times}}\text{,}%
\]
where the units $A^{\times}$ are diagonally mapped to the id\`{e}les.
\end{enumerate}
\end{theorem}

\begin{remark}
The reduced norm one subgroup $J^{1}(A)$ always contains the commutator
subgroup of $J(A)$, explaining why both quotients on the right side are abelian.
\end{remark}

\begin{remark}
The theorem generalizes the main result of \cite{kthyartin}: If $F$ is a
number field, $J(F)$ simplifies to the usual id\`{e}les, the reduced norm one
subgroup $J^{1}(F)$ is trivial, so the second claim simplifies to $K_{1}(\mathsf{LCA}_{F})\cong J(F)/F^{\times}$,
which is Chevalley's id\`{e}le class group and the automorphic side of
global class field theory in dimension one. This statement was the original
prediction made by Clausen in \cite{clausen}.
\end{remark}

In higher degrees, we obtain the following, also matching the picture from the
number field case.

\begin{theorem}
\label{thm_ComputeKn}Let $A$ be a finite-dimensional semisimple $\mathbb{Q}%
$-algebra. Let $\mathfrak{A}\subset A$ be any $\mathbb{Z}$-order.%
\[
K_{n}(\mathsf{LCA}_{A,ab})\cong\left\{  \left.  (\alpha_{P})_{P}\in\prod
_{p \in \Pl}K_{n}(A_{p})\right\vert
\begin{array}
[c]{l}%
\alpha_{P}\in\operatorname*{im}K_{n}(\mathfrak{A}_{p})\text{ for all but
finitely}\\
\text{many }p \in \Plfin
\end{array}
\right\}  \text{,}%
\]
where $\mathfrak{A}_{p}\coloneqq\mathfrak{A}\otimes_{\mathbb{Z}}\mathbb{Z}_{p}$. The
right-hand side is independent of the choice of the order $\mathfrak{A}$ (see \ref{lcw1}, where we had introduced the above notation for restricted products of $K$-groups).
\end{theorem}

\begin{proof}
[Proof of both theorems]\textit{(Step 1)} Use Theorem \ref{thm_FiberSeq} and
Theorem \ref{thm_KAdelicBlocks} for $\mathbb{K}$ being non-connective
$K$-theory. This is a localizing invariant (in fact the universal one by
\cite{MR3070515}) and commutes with products by \cite{kwinfproducts}. However,
$A$ and each $A_{p}$ are regular, so their non-connective $K$-theory agrees
with Quillen's $K$-theory (By \cite[Remark 9.33]{MR3070515} there is an equivalence of the $0$-connective cover of connective vs. non-connective $K$-theory, and \cite[Remarks 3 and 7]{MR2206639} show that all negative degree homotopy groups of non-connective $K$-theory vanish and that it further is an equivalence on $\pi_0$. Thus, it is an equivalence on the whole). We thus get the exact sequence of abelian groups%
\begin{equation}
\cdots\longrightarrow K_{1}(A)\longrightarrow\left.  \prod\nolimits^{\prime
}\right.  K_{1}(A_{p})\longrightarrow K_{1}(\mathsf{LCA}_{A})\longrightarrow
K_{0}(A)\longrightarrow\cdots\text{,}\label{lwwaa2s}%
\end{equation}
where $p$ runs over all places of $\mathbb{Q}$.

This finishes the proof of Theorem \ref{thm_ComputeKn}, so henceforth assume
$n=1$. Then the group in Equation \eqref{lcw1} is also called $JK_{1}(A)$ (e.g., in \cite[p. 2,
end of page]{MR0447211} or \cite[(49.16), p. 224]{MR892316}), so we have
recovered a classical invariant with our $K$-theory computations:%
\begin{equation}
K_{1}(\mathsf{LCA}_{A,ab})\cong JK_{1}(A)\text{.}\label{lmit1}%
\end{equation}
In order to relate this to $J(A)$, we follow classical ideas of Wilson
\cite{MR0447211}:\ there is a fundamental commutative diagram%
\begin{equation}%
\xymatrix{
J(A) \ar@{->>}[r] \ar[dr]_{\operatorname{nr}} & JK_1(A) \ar[d]^{\operatorname
{nr}}_{\cong} \\
& J(\zeta(A))\text{,}
}
\label{lwwaa3a}%
\end{equation}
which we describe now. The group $J(A)$ is Fr\"{o}hlich's id\`{e}le class group of
Equation \eqref{l_def_idelegroupplain}. The horizontal arrow is induced from the
map $A_{p}^{\times}\rightarrow K_{1}(A_{p})$. It is surjective for each $p$ by
\cite[Theorem~40.31]{MR892316}. The downward arrow is induced for every
place $p$ by the reduced norm map%
\[
\operatorname*{nr}\colon K_{1}(A_{p})\overset{\sim}{\longrightarrow}\zeta
(A_{p})^{\times}%
\]
taking values in the units of the center $\zeta(A_{p})$. That these are
isomorphisms is shown for central simple algebras in \cite[Theorem~45.3]{MR892316} whose center is a finite field extension of $\mathbb{Q}%
_{p}$ resp. $\mathbb{R}$ (i.e. is a local field). However, then the general
case follows by the Artin--Wedderburn decomposition and Morita invariance of
$K$-theory on the left side, and on the right side since $\zeta(R)\overset
{\sim}{\longrightarrow}\zeta(M_{n}(R))$ by constant diagonal matrices. The
commutativity of Diagram \ref{lwwaa3a} implies that both arrows originating
from $J(A)$ have the same kernel (which for the diagonal arrow is the reduced
norm one id\`{e}les of Equation \eqref{l_red_norm_one_ideles}),%
\[
J(A)/J^{1}(A)\overset{\sim}{\longrightarrow}JK_{1}(A)\text{.}%
\]
Alongside Equation \eqref{lmit1}, we have proven the first claim.\newline%
\textit{(Step 2)} Since $A$ and each $A_{p}$ are semisimple, one has a
canonical isomorphism $K_{0}(A)\cong\bigoplus_{W}\mathbb{Z}$, where $W$ is the
set of isomorphism classes of simple objects in $\operatorname*{proj}(A)$ (and
analogously for $A_{p}$). Thus, the kernel $K_{0}(A)\rightarrow\left.
\prod\nolimits^{\prime}\right.  K_{0}(A_{p})$ which would occur at the right
end of Sequence \ref{lwwaa2s} is zero (it suffices to check what happens to
each simple module generator of $K_{0}(A)$ and since the underlying functor is
just tensoring $P\mapsto P\otimes_{A}A_{p}$, none goes to zero). Thus, we
obtain%
\[
K_{1}(\mathsf{LCA}_{A})\cong\left.  JK_{1}(A)\right/  \operatorname*{im}%
K_{1}(A)\text{.}%
\]
and we need to understand how $K_{1}(A)$ is mapped to $JK_{1}(A)$ here. To
this end, return to Diagram \ref{lwwaa3a} and observe that this map agrees with the horizontal
arrow. It follows that the map is indeed just the diagonal embedding. Our
second claim follows.
\end{proof}

\newcommand{\Places}{\Pl}
\newcommand{\Primes}{\Plfin}

\section{The \texorpdfstring{$K$}{K}-theory of the category of adelic blocks}

In our main theorem (Theorem \ref{theorem:MainIntroduction}), we describe the $K$-theory spectrum of $\mathsf{LCA}_A$ via a fiber sequence $K(A)\longrightarrow K(\mathsf{LCA}_{A,ab})\longrightarrow K(\mathsf{LCA}_{A}).$  In this sequence, we use the equivalence $\mathsf{LCA}_{A,ab} \simeq \mathsf{J}_{A}^{(\infty)}$ to construct a fiber sequence 
\[\bigoplus_{p \in \Primes}K(\tor \fA_p) \to \prod_{p \in \Places} K(\smod \fA_p) \to K(\mathsf{LCA}_{A,ab})\]
describing the $K$-theory spectrum of $K(\mathsf{LCA}_{A,ab})$.  Here, $\fA_p$ (for $p \in \Pl_{\mathbb{Q}}$) is as introduced in \S\ref{section:Recollections} and $\tor \fA_p$ stands for the subcategory of $\smod \fA_p$ consisting of the $\bZ_p$-torsion modules.

Thus, choose a maximal order $\mathfrak{A}\subset A$. For any finite set of prime numbers $S \subset \Primes$, consider the following category:
\[
\mathsf{L}_{A}^{(S)}\coloneqq\operatorname*{mod}(A_{\mathbb{R}})\times\prod_{p\notin
S}\operatorname*{mod}(\mathfrak{A}_{p})\times\prod_{p\in S}%
\operatorname*{mod}(A_{p})\text{.}\]
The category $\mathsf{L}_{A}^{(S)}$ is a product of abelian categories and, as such, is an abelian category itself.  For any inclusion of finite sets $S\subseteq S^{\prime}$ of primes there is an exact functor $\mathsf{U}^{(S)}\rightarrow\mathsf{U}^{(S^{\prime})}$ induced by $(-)\mapsto(-)\otimes_{\mathfrak{A}_{p}}A_{p}$ termwise for all primes $p\in S^{\prime}\setminus S$, and the identity functor for the remaining $p$.
Define
\[
\mathsf{L}_{A}^{(\infty)}\coloneqq\left.  \underset{S}{\operatorname*{colim}}\left.
\mathsf{L}_{A}^{(S)}\right.  \right.  \text{.}%
\]
We write $\T$ for the category for the full subcategory of $\mathsf{L}_{A}^{(\varnothing)}$ consisting of those objects which are finite direct sums of torsion $\fA_p$-modules.  Put differently, as sequences indexed by the set of places, the objects are almost everywhere zero and contain a torsion $\fA_p$-module for finitely many primes $p$, or thus:
\[\T \coloneqq \oplus_{p \in \Primes} \tor(\fA_p) = \left.  \underset{S}{\operatorname*{colim}}\left.
\prod_{p \in S}\tor(\fA_p)\right.  \right.  \text{,} \]
where $S$ ranges over the finite subsets of $\Plfin_{,\mathbb{Q}}$.  We have the following lemma.

\begin{lemma}\label{lemma:AboutLS}Let $S$ be a finite set of primes.
\begin{enumerate}
	\item $\mathsf{L}_{A}^{(S)} \simeq \mathsf{L}_{A}^{(S)} / \oplus_{p \in S} \tor(\fA_p).$
	\item The category $\mathsf{L}_{A}^{(S)}$ has enough projectives.
	\item The essential image of the embedding $\mathsf{J}_{A}^{(S)} \to \mathsf{L}_{A}^{(S)}$ consists of all projective objects in $\mathsf{L}_{A}^{(S)}.$
	\item The global dimension of $\mathsf{L}_{A}^{(S)}$ is one.
\end{enumerate}
\end{lemma}

\begin{proof}
As $A_p = \fA_p \otimes_{\bZ_p} \bQ_p$, for any prime $p$, we have $\mod A_p \simeq \mod \fA_p / \tor \fA_p$.  The first statement follows from this observation.  For the other statements, note that both the projective objects and projective resolutions in $\mathsf{L}_{A}^{(S)}$ can be determined componentwise.  The statements follow easily from this.  
\end{proof}

From this, we find the following properties of $\mathsf{L}_{A}^{(\infty)}.$

\begin{proposition}\label{proposition:AboutLinfty} With notation as above, we have:
\begin{enumerate}
	\item $\mathsf{L}_{A}^{(\infty)} \simeq \mathsf{L}_{A}^{(\varnothing)} / \oplus_p \tor(\fA_p).$  In particular, $\mathsf{L}_{A}^{(\infty)}$ is abelian.
	\item The category $\mathsf{L}_{A}^{(\infty)}$ has enough projectives.
	\item The essential image of the embedding $\mathsf{J}_{A}^{(\infty)} \to \mathsf{L}_{A}^{(\infty)}$ consists of all projective objects in $\mathsf{L}_{A}^{(\infty)}.$
	\item The global dimension of $\mathsf{L}_{A}^{(\infty)}$ is one.
\end{enumerate}
\end{proposition}

\begin{proof}
It follows from Lemma \ref{lemma:AboutLS} that, up to equivalence, for each $p \in S$, we may replace the component $\operatorname*{mod}(A_{p})$ in $\mathsf{L}_{A}^{(S)}$ by $\operatorname*{mod}(\mathfrak{A}_{p}) / \tor(\fA_p).$  In this way, each of the transition maps $\mathsf{L}_{A}^{(S)} \to \mathsf{L}_{A}^{(S^\prime)}$ is a localization (at the Serre subcategory $\oplus_{p \in S^\prime \setminus S} \tor(\fA_p)$).  The first statement now follows easily.

In the rest of the proof, we identify $\mathsf{L}_{A}^{(\infty)}$ with $\mathsf{L}_{A}^{(\varnothing)} / \oplus_p \tor(\fA_p)$, meaning that the objects of $\mathsf{L}_{A}^{(\infty)}$ are the objects of $\mathsf{L}_{A}^{(\varnothing)}.$

To see that $\mathsf{L}_{A}^{(\infty)}$ has enough projectives, it suffices to see that the objects $P = (P_p)_{p}\in \mathsf{L}_{A}^{(\infty)}$ for which every $P_p \in \smod (\fA_p)$ is projective, are projective in $\mathsf{L}_{A}^{(\infty)}.$  This is straightforward.

Now, let $P = (P_p)_{p} \in \mathsf{L}_{A}^{(\infty)}$ be any projective object.  As $\mathsf{L}_{A}^{(\varnothing)}$ has enough projectives, we may consider an epimorphism $\alpha\colon Q \to P$ in $\mathsf{L}_{A}^{(\varnothing)}$ with $Q$ projective in $\mathsf{L}_{A}^{(\varnothing)}$ (thus, $Q \in \mathsf{J}_{A}^{(\varnothing)}$).  As $P$ is projective in $\mathsf{L}_{A}^{(\infty)}$, there is a section $\gamma^{-1} \circ \beta\colon P \to Q$ in $\mathsf{L}_{A}^{(\infty)}$.  Hence, there is a finite set of primes $S$ such that $\gamma^{-1} \circ \beta \in \Hom_{\mathsf{L}_{A}^{(S)}}(P,Q)$ and $\alpha \circ \gamma^{-1} \circ \beta = 1_P.$  Hence, $P$ is projective in $\mathsf{L}_{A}^{(S)}$ and so $P \in \mathsf{J}_{A}^{(S)}$ by Lemma \ref{lemma:AboutLS}.  This shows that $P \in \mathsf{J}_{A}^{(\infty)}$, as required.

Finally, as projective resolutions can be taken componentwise, we see that the global dimension of $\mathsf{L}_{A}^{(\infty)}$ is at most one.  As not all objects are projective, for example an object $X = (X_p)_p$ where each $X_p \in \smod(\fA_p)$ is non-projective, the global dimension is not zero.
\end{proof}

\begin{remark}
As $\mathsf{J}_{A}^{(\infty)}$ is equivalent to the category of projectives of an abelian category of global dimension at most one, every morphism in $\mathsf{J}_{A}^{(\infty)}$ has a kernel and all kernels are split monomorphisms.  This gives an alternative to steps 2 and 3 of the proof of proposition \ref{prop_LCAAbHasKernelsAndCokernels}.
\end{remark}

%


In the next proposition, we write $\DTb(\mathsf{L}_{A}^{(\infty)})$ for the full subcategory of $\Db(\mathsf{L}_{A}^{(\varnothing)})$ consisting of those objects whose cohomology groups lie in $\T = \oplus_{p \in \Primes} \tor(\fA_p).$

\begin{proposition}\label{proposition:KernelCoincides}
The natural triangle functor $\Db(\T) \to \DTb(\mathsf{L}_{A}^{(\varnothing)})$ is an equivalence.
\end{proposition}

\begin{proof}
We show that the conditions of \cite[Theorem 13.2.8]{KashiwaraSchapira06} are satisfied, namely that for each monomorphism $f\colon X \to Y$ in $\mathsf{L}_{A}^{(\varnothing)}$ with $X \in \T$, there is a morphism $Y \to Y'$ with $Y' \in \T$ such that the composition $X \to Y \to Y'$ is a monomorphism.  Thus, let $f\colon X \to Y$ be a monomorphism in $\mathsf{L}_{A}^{(\varnothing)}$ with $X \in \T$.  Recall that each object of $L = (L_p)_{p \in \Places}$ of $\mathsf{L}_{A}^{(\infty)}$ is a string of objects, where each $L_p \in \mod \fA_p$.  As $X \in \T$, there is a finite set $S \subset \Primes$ such that $X_p \in \tor \fA_p$ for $p \in S$ and $X_p = 0$ for $p \notin S$.

Let $Y'$ be the direct summand of $Y$ given by $Y_p' = Y_p$ when $p \in S$ and $Y'_p = 0$ otherwise.  The composition $X \to Y \to Y'$ is still a monomorphism.

It now follows from Lemma \ref{lemma_StructOfFinGenModulesOverMaxOrder} that $Y' = Y'_{\text{proj}} \oplus Y'_{\text{tor}}$ with $Y'_{\text{proj}}$ projective and $Y'_{\text{tor}}$ torsion.  As $\Hom(X, Y'_{\text{proj}}) = 0$, the composition $X \to Y \to Y' \to Y'_{\text{tor}}$ is a monomorphism with codomain in $\T.$  The statement now follows from \cite[Theorem 13.2.8]{KashiwaraSchapira06}.
\end{proof}

The following theorem gives the required fiber sequence.  It is an easy consequence of the results in this section.

\begin{theorem}
Let $\mathbb{K}\colon\operatorname*{Cat}_{\infty}^{\operatorname*{ex}}\longrightarrow\mathsf{A}$ be a localizing invariant with values in a stable presentable $\infty$-category $\mathsf{A}$. If $\mathbb{K}$ commutes with countable products, then there is a fiber sequence
\[\bigoplus_{p \in \Primes} \bK(\tor \fA_p) \to \prod_{p \in \Places} \bK(\smod \fA_p) \to \bK(\mathsf{LCA}_{A,ab}).
\]
\end{theorem}

\begin{proof}
The localization sequence $\oplus_{p} \tor (\fA_p) \to \prod_p \smod (\fA_p) \to \mathsf{L}_{A}^{(\infty)}$ of abelian categories, given by Proposition \ref{proposition:AboutLinfty}, yields an exact sequence $\Dbinf(\oplus_{p} \tor (\fA_p)) \to \Dbinf(\prod_p \smod (\fA_p)) \to \Dbinf(\mathsf{L}_{A}^{(\infty)})$ of stable $\infty$-categories (in the sense of \cite{LurieHA}); this uses that the natural functor $\Db(\oplus_{p} \tor (\fA_p)) \to \DTb(\mathsf{L}_{A}^{(\infty)})$ is a triangle equivalence, see \cite[Corollary 5.11]{MR3070515}, or \cite{MR3904731}. Furthermore, the natural embedding $\mathsf{J}_{A}^{(\infty)} \to \mathsf{L}_{A}^{(\infty)}$ induces an equivalence $\Dbinf(\mathsf{J}_{A}^{(\infty)}) \to \Dbinf(\mathsf{L}_{A}^{(\infty)})$. Indeed, since $\mathsf{J}_{A}^{(\infty)}$ is the category of projectives of $\mathsf{L}_A^{(\infty)}$, and the latter has enough projectives (see Proposition \ref{proposition:AboutLinfty} for both claims), the equivalence follows from \cite[Remark 3.15]{HenrardvanRoosmalen20Preresolving}.
This leads to an exact sequence $\Dbinf(\oplus_{p} \tor (\fA_p)) \to \Dbinf(\prod_p \smod (\fA_p)) \to \Dbinf(\mathsf{J}_{A}^{(\infty)})$ of stable $\infty$-categories.  Applying a localizing invariant as in the statement of the theorem gives the required sequence. Finally, use Proposition \ref{prop_IdentifyJ}.
\end{proof}

\begin{remark}
For a weakly localizing invariant $\bK$ which need not commute with countable products, there is a fiber sequence $\bigoplus_{p\in \Plfin} \bK(\tor \fA_p) \to \bK(\prod_{p \in \Pl} \smod \fA_p) \to \bK(\mathsf{LCA}_{A,ab}).$
\end{remark}


\section{A corollary regarding the second \texorpdfstring{$K$}{K}-group\label{section_HilbertReciprocityClassical}}

Suppose $G$ is an abelian group. We denote its profinite completion by%
\[
G^{\wedge}\coloneqq\underset{N}{\underleftarrow{\lim}}\,G/N\text{,}%
\]
where the limit is taken over all finite-index subgroups $N\subseteq G$,
partially ordered by inclusion. This notation is not to be confused with the
Pontryagin dual $G^{\vee}$.

We shall only need $G^{\wedge}$ as an abelian group and ignore that one could
consider it as a profinite topological group.

One can formulate the above in a more categorical fashion: the forgetful functor from profinite groups to groups has profinite completion as its left adjoint.

In particular, the unit of this adjunction provides a natural map
$G\rightarrow G^{\wedge}$ to the profinite completion. Despite the name
`completion', this map need not be injective.

\begin{example}\begin{enumerate}
  \item If $G$ is a finite group, the map $G\rightarrow G^{\wedge}$ is an isomorphism.
	\item We have $\mathbb{Z}^{\wedge}\cong\prod_{p}\mathbb{Z}_{p}$, where
$\mathbb{Z}_{p}$ denotes the $p$-adic integers.
  \item As $\mathbb{Q}$ has no nonzero finite quotient groups, we have $\mathbb{Q}^{\wedge}=0$.  Similarly, we have $\mathbb{R}^{\wedge}=0$, and $(\mathbb{Q}/\mathbb{Z})^{\wedge}=0$.
\end{enumerate}
\end{example}

\begin{construction}
\label{construction_AbstractHReciprocityLaw}Let $A$ be a simple
finite-dimensional $\mathbb{Q}$-algebra with center $F$. Below, we construct
an exact sequence%
\begin{equation}
K_{2}(A)\longrightarrow\left.  \underset{v\,\smallskip\,\smallskip
\,\smallskip}{\prod\nolimits^{\prime}}\right.  K_{2}(A_{v})\longrightarrow
K_{2}(\mathsf{LCA}_{A})\longrightarrow0\text{,} \label{lvppa5}%
\end{equation}
where $v$ runs over the set $\mathcal{P}\ell_{F}$. The middle term is to be
understood in the meaning of Equation \eqref{lcw1}.\medskip\newline The
construction only relies on Theorem \ref{theorem:MainIntroduction}, applied
with non-connective $K$-theory as the input invariant, and uses no class field theory.
\end{construction}

\begin{proof}
[Construction steps] Apply Theorem \ref{theorem:MainIntroduction} to $A$ using
non-connective $K$-theory as the localizing invariant. This outputs a fiber
sequence of spectra and its long exact sequence of homotopy groups reads as
follows%
\begin{equation}
\cdots\rightarrow K_{2}(A)\rightarrow\left.  \underset{p\,\,\,}{\prod
\nolimits^{\prime}}\right.  K_{2}(A_{p})\rightarrow K_{2}(\mathsf{LCA}%
_{A})\overset{\alpha}{\rightarrow}K_{1}(A)\overset{\beta}{\rightarrow}\left.
\underset{p\,\,\,}{\prod\nolimits^{\prime}}\right.  K_{1}(A_{p})\rightarrow
\cdots\text{,} \label{lmvp2}%
\end{equation}
where $p$ runs through $\mathcal{P}\ell_{\mathbb{Q}}$ and where we have
labeled various arrows for later reference. Here we use that for $n\geq1$
non-connective $K$-groups $K_{n}(-)$ agree with ordinary Quillen $K$-groups
(\cite[Remarks 3 and 7]{MR2206639}). For any finite place $p\in\mathcal{P}%
\ell_{\mathbb{Q}}$ we have%
\begin{equation}
A_{p}=A\otimes_{\mathbb{Q}}\mathbb{Q}_{p}=\bigoplus_{v\mid p}A_{v}\text{,}
\label{lmvp1}%
\end{equation}
where $v$ runs through the finitely many places $v$ of $F$ lying over $p$.
This also holds analogously for $p=\mathbb{R}$. Hence, $K_{i}(A_{p}%
)\cong\bigoplus_{v\mid p}K_{i}(A_{v})$. Finally, the morphism $\beta$ is
injective (it suffices to check that the map to the localization at a single prime $p$ of our choice is injective. In detail: apply Remark
\ref{rmk_ReducedNormInjectivities} to each simple summand of $A$, regarded over its individual center). Thus, we must have
$\alpha=0$ and we can truncate the exact sequence at this point. We arrive at
Equation \eqref{lvppa5}.
\end{proof}

\begin{remark}
\label{rmk_ReducedNormInjectivities} Suppose $A$ is a central simple $F$-algebra. Then $A_v$ is a central simple $F_v$-algebra. The reduced norm respects any base change, so%
\[
\xymatrix{
K_i(A) \ar[r] \ar[d]_{\operatorname{nr}} & K_i(A_v) \ar[d]^{\operatorname{nr}}
\\
K_i(F) \ar[r] & K_i(F_v)
}
\]
commutes for $i=0,1,2$ (and just for these $i$ because the reduced norm only
exists in this range). Suppose $i=1$. The reduced norm $K_1(A)\rightarrow K_1(F)$ is injective (\cite[Theorem 7.51]{MR1038525} or \cite[Theorem 45.3]{MR892316}). The bottom
horizontal arrow is just $F^{\times}\rightarrow F_{v}^{\times}$ and thus also injective
as this is just the metric completion at $v$. It follows that $K_{1}%
(A)\rightarrow K_{1}(A_{v})$ is injective. If $v$ is a finite place, the reduced
norm $K_{1}(A_{v})\rightarrow K_{1}(F_{v})$ is also an isomorphism.
\end{remark}

\section{Hilbert Reciprocity Law (classical)\label{subsect_ClassHRL}}

In this section we shall freely use the Hilbert symbol and its properties. A
useful survey reference is \cite[Chapter II, \S 7]{MR1941965}, and proofs of
the key properties can be found in \cite[Chapter IV, \S 5.1]{MR1915966}.

Let $F$ be a number field. We write $m\coloneqq\#\mu(F)$ for the number of its roots
of unity. Similarly, write $m_{v}\coloneqq\#\left(  \mu(F_{v})^{\wedge}\right)  $ for
any place $v\in\mathcal{P}\ell_{F}$. Thus, $m_{v}\geq1$ is a finite number in
any case. See Remark \ref{rmk_AtComplexPlaces} for an explanation why this
holds and an alternative, perhaps more common, description of the value
$m_{v}$.

The Hilbert Reciprocity Law states that for any $\alpha,\beta\in F^{\times}$
we have the identity%
\begin{equation}
\prod_{v\in\mathcal{P}\ell_{F}}h_{v}(\alpha,\beta)^{\frac{m_{v}}{m}}=1\text{,}
\label{lzzp0}%
\end{equation}
where%
\begin{equation}
h_{v}\colon F_{v}^{\times}\times F_{v}^{\times}\longrightarrow\mu(F_{v})
\label{lzzp1}%
\end{equation}
is the \emph{Hilbert symbol} for the local field $F_{v}$, the completion of
$F$ at the place $v$. All but finitely many factors in the product are $1$.

It does not make a difference whether the product over $v\in\mathcal{P}%
\ell_{F}$ includes the complex places or not because for these the Hilbert
symbol is trivial by construction\footnote{We recall the construction of the Hilbert symbol in
Equation \eqref{lzzu3} below. The Hilbert symbol concerns the action of the
Artin symbol on radical extensions of the local field. As $\mathbb{C}$ is
algebraically closed, no non-trivial extensions exist, so the action is
necessarily by the identity map. For the same reason, the action at real
places is necessarily $2$-torsion.}.

The Hilbert Reciprocity Law is a consequence of the stronger Artin Reciprocity Law (see, for example, \cite{Shastri00}), but is still a result with many important applications on its own.  For example, it immediately implies Gauss' original Quadratic Reciprocity Law.

\begin{example}
Suppose $F=\mathbb{Q}$. Then for all primes $p\neq2$ the Hilbert symbol%
\[
h_{p}\colon\mathbb{Q}_{p}^{\times}\times\mathbb{Q}_{p}^{\times}\longrightarrow
\mu(\mathbb{Q}_{p})\cong\mathbb{F}_{p}^{\times}%
\]
agrees with the tame symbol. We have $\frac{m_{p}}{m}=\frac{p-1}{2}$ and
therefore only the image of $h_{p}$ in the quotient group $\mathbb{F}%
_{p}^{\times}/\mathbb{F}_{p}^{\times2}\cong\{\pm1\}$ matters for the
reciprocity law. This makes it possible to express $h_{p}$ in terms of
Legendre symbols, leading to the classical formulation of Quadratic
Reciprocity. Only for $p=2,\infty$ the Hilbert symbol differs from the tame
symbol: for $p=2$ the residue field $\mathbb{F}_{2}^{\times}$ is trivial and
so is the tame symbol, for $p=\infty$ the definition of the tame symbol
collapses as there is no valuation ideal.
\end{example}

Moore \cite{MR244258} later found a strengthening of the\ Hilbert Reciprocity
Law. We recall his way of phrasing it. As a first step, one uses that
the\ Hilbert symbol satisfies the Steinberg relation%
\[
h_{v}(\alpha,1-\alpha)=1\qquad\text{for all}\qquad\alpha\in F\setminus\{0,1\}
\]
and therefore, using Milnor's formula, namely%
\begin{equation}
K_{2}(F)\cong\left(  F^{\times}\otimes_{\mathbb{Z}}F^{\times}\right)
/(\text{Steinberg relation})\text{,} \label{lmilnorformula}%
\end{equation}
can be reinterpreted as a map%
\[
h_{v}\colon K_{2}(F_{v})\longrightarrow\mu(F_{v})
\]
replacing Equation \eqref{lzzp1}. Then Moore proves the following.

\begin{theorem}
[Moore sequence] \label{theorem:MooreSequence} Let $F$ be a number field. Then the sequence%
\[
K_{2}(F)\longrightarrow\bigoplus_{v\text{ }\operatorname*{noncomplex}}%
\mu(F_{v})\overset{\cdot\frac{m_{v}}{m}}{\longrightarrow}\mu(F)\longrightarrow
0
\]
is exact, where the first arrow is the Hilbert symbol, and the second is
taking the $\frac{m_{v}}{m}$-th power.
\end{theorem}

The statement that this is a complex agrees with Hilbert's Reciprocity Law.
The exactness is Moore's refinement. A nice presentation and proof of Moore's
theorem can be found in \cite{MR311623}.

The new perspective brought by Moore's formulation directly suggests
numerous follow-up questions. Is there a natural way to extend Moore's
sequence to the left? Does the sequence have a counterpart for $K_{n}$ with
$n\neq2$?

Below, we present a novel viewpoint which might be considered an answer to
these questions. Building on top of this angle, we later go on to replace the
number field by a possibly non-commutative simple algebra.

Write $\mathcal{O}_{v}\coloneqq\mathcal{O}_{F,v}$ as a shorthand for
the valuation ring and $\kappa(v)$ for its residue field.

\begin{lemma}\label{lemma:ConstructionOfT}
For any place $v \in \Pl_F$, let $h_v\colon K_2(F_v) \to \mu(F_v)$ be the Hilbert symbol.
\begin{enumerate}
	\item For any complex place $v$, the group $K_2(F_v)$ is divisible.
	\item For noncomplex places $v$, we have
	\begin{enumerate}
	\item the maps $h_v\colon K_2(F_v) \to \mu(F_v)$ are epimorphisms with uniquely divisible kernels,
	\item for all finite places, the maps $h_v\colon K_2(\OO_v) \to \mu(F_v)[p^{\infty}]$ are epimorphisms with uniquely divisible kernels, and
	\item for all but finitely many places, we have $h_v(\OO_v) = 0$.
	\end{enumerate}
\end{enumerate}
\end{lemma}

\begin{proof}
We start by considering a complex place $v$.  In this case, $F_{v}\cong\mathbb{C}$ is algebraically closed and by $K_{2}(\mathbb{C})/lK_{2}(\mathbb{C})\cong H^{2}(\mathbb{C},\mu_{l}^{\otimes2})=0$ for all $l\geq1$ (or a direct computation using Milnor $K$-groups, $\{a,b\}=l\{a,\sqrt[l]{b}\}$) one sees that $K_{2}(\mathbb{C})$ is a divisible group.

Consider now the case where $v$ is noncomplex.  Using local class field theory, Moore \cite{MR244258} has shown the exactness of the sequence%
\[
0\longrightarrow m_{v}K_{2}(F_{v})\longrightarrow K_{2}(F_{v})\overset{h_{v}%
}{\longrightarrow}\mu(F_{v})\longrightarrow0\text{,}%
\]
where $h_{v}$ is the Hilbert symbol. This proceeds as follows: firstly, one
uses local class field theory to construct the Hilbert symbol, using the
explicit presentation of $K_{2}$ through Equation \eqref{lmilnorformula},%
\begin{equation}
h_{v}(a\otimes b)\coloneqq\frac{\operatorname*{Art}(a)\sqrt[m_{v}]{b}}{\sqrt[m_{v}%
]{b}}\in\mu(F_{v})\qquad\text{for}\qquad a,b\in F_{v}^{\times}\text{,}%
\label{lzzu3}%
\end{equation}
where $\sqrt[m_{v}]{b}$ is any $m_{v}$-th root of $b\in F_{v}^{\times}$ and
$\operatorname*{Art}\colon F_{v}^{\times}\rightarrow\operatorname*{Gal}%
(F^{\operatorname*{ab}}/F)$ denotes the local Artin map. Secondly, one
verifies that $h_{v}$ respects the Steinberg relation, so it factors over
$K_{2}(F_{v})$, and then that $h_{v}$ is zero on $m_{v}K_{2}(F_{v})$, which is
pretty immediate. Moore then goes on to prove that $m_{v}K_{2}(F_{v})$ is
divisible, which Merkurjev \cite{MR717828} strengthened to being uniquely
divisible. A complete textbook proof of all these results with full details
can be found in \cite[Ch. IX, \S 4]{MR1915966}.

Suppose that $v$ is a finite place. 
The torsion group $\mu(F_{v})$ splits into its different $\ell$-primary torsion parts,%
\[
\mu(F_{v})\cong\mu(F_{v})[p^{\infty}]\oplus\mu(F_{v})[\text{prime-to-}%
p^{\infty}]\text{.}%
\]
The group $\mu(F_{v})[p^{\infty}]$ denotes the $p$-primary part, where $p$ is
the prime in $\mathbb{Z}$ lying below the finite place $v$. Its complement,
the prime to $p$ part, is known to agree with $\mu(\kappa(v))$ through the
(canonical) Teichm\"{u}ller lift\footnote{See \cite[Chapter I, \S 7.1-7.3]%
{MR1915966} for a detailed construction of the Teichm\"{u}ller
representatives. One can also define a map $\omega\colon\kappa(v)^{\times
}\rightarrow\mu(F_{v})[$prime-to-$p^{\infty}]$ by%
\[
\omega(x)\coloneqq\lim_{n\rightarrow\infty}\tilde{x}^{q^{n}}\text{,}%
\]
where $q\coloneqq\#\kappa(v)$ and $\tilde{x}$ is any preimage of $x$ in
$\mathcal{O}_{v}$. Once one shows that the limit exists and is independent of
the choice of $\tilde{x}$, it is clear that $\omega$ is a group homomorphism
and a section to $\mathcal{O}_{v}^{\times}\rightarrow\kappa(v)^{\times}$.}.
The quotient sequence of abelian categories%
\[
\operatorname{mod}_{\mathfrak{m}_{v}}(\mathcal{O}_{v})\longrightarrow
\operatorname{mod}(\mathcal{O}_{v})\longrightarrow\operatorname{mod}(F_{v})
\]
induces a long exact sequence in $K$-theory (by the localization theorem;
$\operatorname{mod}_{\mathfrak{m}_{v}}(\mathcal{O}_{v})$ denotes finitely
generated modules supported on the prime $\mathfrak{m}_{v}$). We get%
\[
\cdots\longrightarrow K_{2}(\mathcal{O}_{v})\longrightarrow K_{2}%
(F_{v})\overset{\partial}{\longrightarrow}K_{1}(\kappa(v))\longrightarrow
\cdots\text{,}%
\]
where $\partial$ is the tame symbol, see \cite[Lemma~9.12]{Srinivas08}, and $K_1(\kappa(v)) \cong \mu(F_{v})[\text{prime-to-}p^{\infty}]$.  One can replace \textquotedblleft$\cdots$\textquotedblright\ on either side by zero and still get an exact sequence. On the right, this follows because
$\partial$ is split by (using Milnor $K$-group notation)%
\[
x\mapsto\{\pi,\omega(x)\}\text{,}%
\]
where $\omega\colon\kappa(v)^{\times}\rightarrow F_{v}^{\times}$ is the
Teichm\"{u}ller lift and $\pi$ any (fixed once and for all) uniformizer. On
the left, because $K_{2}$ of a finite field vanishes.

As the tame symbol factors as $K_2(F_v) \xrightarrow{h_v} \mu(F_v) \to \mu(F_{v})[\text{prime-to-}
p^{\infty}]$, where the last map is the canonical projection, we find the following commutative diagram
\[\begin{tikzcd}
0 \ar[r] & m_v K_2(F_v) \ar[r] \ar[d] & K_2(F_v) \ar[r, "h_v"] \ar[d, equal] & \mu(F_v) \ar[d] \ar[r]& 0\\
0 \ar[r] & K_2(\OO_v) \ar[r]&  K_2(F_v) \ar[r, "\partial"] & \mu(F_{v})[\text{prime-to-}p^{\infty}] \ar[r] & 0
\end{tikzcd}\]
with exact rows.  The left downward arrow is a split monomorphism since $m_v K_2(F_v)$ is divisible (and hence injective).  The Snake Lemma now gives a (split) exact sequence
\[0 \to m_v K_2(F_v) \to K_2(\OO_v) \xrightarrow{h'_v} \mu(F_{v})[p^{\infty}] \to 0\text{,}\]
where $h'_v$ the Hilbert symbol, with the domain restricted to $K_2(\OO_v) \subseteq K_2(F_v)$ and the codomain restricted to $\mu(F_{v})[p^{\infty}] \subseteq \mu(F_v)$.   We note that $F_{v}$ only has non-trivial $p$-primary roots of unity if $F_{v}/\mathbb{Q}_{p}$ is ramified. This only happens at those primes where $F/\mathbb{Q}$ is ramified, and thus finitely many.  Hence, $\mu(F_{v})[p^{\infty}]=0$ for all but finitely many places $v$.
\end{proof}

\begin{construction}\label{construction:ConstructionofT}
We now consider a map $T\colon K_2(\mathsf{LCA}_{F,ab}) \to \bigoplus_{v\enskip\operatorname{noncomplex}}{\mu(F_v)}$ as the composition of the maps $K_2(\mathsf{LCA}_{F,ab}) \to \prod_v' K_{2}(F_{v}) \to \bigoplus_{v\enskip\operatorname{noncomplex}}{\mu(F_v)}$, given as follows.
\begin{enumerate}
	\item Applying \Cref{theorem:MainIntroduction} with non-connective $K$-theory as the localizing invariant, we find an isomorphism $K_2(\mathsf{LCA}_{F,ab}) \xrightarrow{\cong} \prod_v' K_{2}(F_{v})$.  The isomorphism is induced by the equivalence $\mathsf{LCA}_{F,ab}\overset{\simeq}{\longrightarrow}\mathsf{J}_{F}^{(\infty)}$ from \Cref{prop_IdentifyJ}.
  \item For each place $v$, we consider the Hilbert symbol $h_v\colon K_2(F_v) \to \mu(F_v)$.  This then induces a map $\prod_v h_v\colon \prod_v K_2(F_v) \to \prod_v \mu(F_v)$, which restricts to a map
	\[\prod\nolimits_v' K_2(F_v) \to \bigoplus_{v\enskip\operatorname{noncomplex}}{\mu(F_v)}\]
	by the last claim of \Cref{lemma:ConstructionOfT}.
\end{enumerate}
We remind the reader that
	\[ \prod\nolimits_v' K_{2}(F_{v}) = \left\{  \left.  (\alpha_{v})_{v}\in
\prod_{v\in\Pl_{F}}K_{2}(F_{v})\right\vert
\begin{array}
[c]{l}%
\alpha_{v}\in K_{2}(\mathcal{O}_{v})\text{ for all but}\\
\text{finitely many finite places }v
\end{array}
\right\} \]
\end{construction}

\begin{theorem}
\label{thm_numfieldsit}Let $F$ be a number field.

\begin{enumerate}
\item\label{enumerate:Moore:1} The diagram%
\begin{equation}%
\xymatrix{
K_2(F) \ar@{=}[d] \ar[r] & K_2(\mathsf{LCA}_{F,ab}) \ar[r]
\ar[d]^{T} & K_2(\mathsf{LCA}_{F}) \ar[r] \ar[d]^{\overline{T} } & 0 \\
K_2(F) \ar[r] & \bigoplus_{v\enskip\operatorname{noncomplex}}{\mu(F_v)}
\ar[r] & \mu(F) \ar[r] & 0, \\
}
\label{lzwx1}%
\end{equation}
commutes, where the bottom row is the Moore sequence (\Cref{theorem:MooreSequence}) and the top row comes from Theorem \ref{theorem:MainIntroduction}. The arrow $T$ comes from the Hilbert symbol (see \Cref{construction:ConstructionofT}), and arrow $\overline{T}$ is the induced map on the quotient.

\item\label{enumerate:Moore:2} Both $T$ and $\overline{T}$ are surjective, and their kernels consist of the divisible elements of their respective domains.

\item\label{enumerate:Moore:3} Both rows in Diagram \ref{lzwx1} are exact.

\item\label{enumerate:Moore:4} Taking the profinite completion of the entire diagram, all downward
arrows become isomorphisms.

\item\label{enumerate:Moore:5} $K_{2}(\mathsf{LCA}_{F})^{\wedge}\cong\mu(F)$.
\end{enumerate}
\end{theorem}

As the top row directly comes from the results in this paper, statement
\eqref{enumerate:Moore:4} could be summarized as the following slogan: \textit{Upon profinite
completion, Moore's sequence can be identified with an excerpt of the long
exact sequence coming from the fibration in} Theorem \ref{theorem:MainIntroduction}.

We note that our proof of Theorem \ref{theorem:MainIntroduction}, and thus the fibration sequence
referenced in the slogan, did not rely on any use of class field theory.

\begin{proof}
The top row in Diagram \ref{lzwx1} is exact by Construction \ref{construction_AbstractHReciprocityLaw}. The exactness of the bottom row is Moore's theorem, \cite{MR311623}.  This shows that statement \eqref{enumerate:Moore:3} holds.

For the commutativity of the left square, we observe that the top-right branch is given by 
\[\begin{tikzcd}[row sep = small]
K_2(F) \ar[r] & K_{2}(\mathsf{LCA}_{F,ab}) \ar[r,"\cong"] & \prod_v^{\prime } K_{2}(F_{v}) \ar[r] & \bigoplus_{v\enskip\operatorname{noncomplex}}{\mu(F_v)} \\
\alpha \otimes \beta \ar[rr, mapsto] && (\alpha \otimes \beta)_v \ar[r, mapsto]& \left( h_v(\alpha \otimes \beta) \right)_v \text{.} 
\end{tikzcd}\]
The commutativity of the left square is now immediate from the definition of the first arrow in Moore's sequence.  The commutativity of the right square holds by construction.  This shows statement \eqref{enumerate:Moore:1}.

It follows from \Cref{lemma:ConstructionOfT} that $T$ is an epimorphism.  The commutativity of the right-most square then shows that the composition $K_2(\mathsf{LCA}_{F,ab}) \to K_2(\mathsf{LCA}_{F}) \xrightarrow{\overline{T}} \mu(F)$ is an epimorphism.  We conclude that $\overline{T}$ is an epimorphism as well.

It follows from \Cref{lemma:ConstructionOfT} that $\ker T$ is a uniquely divisible subgroup of $K_{2}(\mathsf{LCA}_{F,ab})$.  As the codomain of $T$ is torsion, the $\ker T$ needs to contain all divisible elements of $K_{2}(\mathsf{LCA}_{F,ab})$.  This shows that $\ker T$ consists of all divisible elements of $K_{2}(\mathsf{LCA}_{F,ab})$.

It follows from the Snake Lemma (or from a straightforward diagram chase) that the canonical map $\ker T \to \ker \overline{T}$ is an epimorphism.  As $\ker T$ is a divisible group, so is $\ker \overline{T}$.  As $\mu(F)$ is torsion, we know that all divisible elements of $K_{2}(\mathsf{LCA}_{F})$ lie in $\ker \overline{T}$.    This establishes statement \eqref{enumerate:Moore:2}.

We now prove statement \eqref{enumerate:Moore:4}.  Consider the exact sequences
\[\begin{tikzcd}[row sep = tiny]
0 \ar[r] &\ker T \ar[r]& K_{2}(\mathsf{LCA}_{F,ab}) \ar[r, "T"]& \bigoplus_{v\enskip\operatorname{noncomplex}}{\mu(F_v)} \ar[r]& 0\text{, and} & \\
0 \ar[r] & \ker \overline{T} \ar[r] & K_{2}(\mathsf{LCA}_{F}) \ar[r] & \mu(F) \ar[r] & 0\text{.}
\end{tikzcd}\]
As both $\ker T$ and $\ker \overline{T}$ are divisible groups (and hence injective), these sequences are split exact.  Since taking profinite completions is an additive functor, these stay (split) exact after profinite completion.  Using that $(\ker T)^\wedge = 0 = (\ker \overline{T})^\wedge$, as profinite completions of divisible groups are zero, we find that $T^\wedge$ and $\overline{T}^\wedge$ are isomorphisms, as required.

Finally, statement \eqref{enumerate:Moore:5} follows from \eqref{enumerate:Moore:4}, together with the observation that $\mu(F) \cong \mu(F)^\wedge$ since $\mu(F)$ is a finite group.
\end{proof}

We also deduce the following characterization:

\begin{corollary}
Up to isomorphisms, the Moore sequence arises from Construction
\ref{construction_AbstractHReciprocityLaw} by quotienting out the subgroups of
divisible elements in the right two terms.%
\[
\xymatrix{
K_2(F) \ar@{=}[d] \ar[r] & K_2(\mathsf{LCA}_{F,ab})_{/div} \ar[r]
\ar[d]^{\cong} & K_2(\mathsf{LCA}_{F})_{/div} \ar[r] \ar[d]^{\cong} & 0 \\
K_2(F) \ar[r] & \bigoplus_{v\enskip\operatorname{noncomplex}}{\mu(F_v)}
\ar[r] & \mu(F) \ar[r] & 0. \\
}
\]
\end{corollary}

This shows that, up to isomorphism, the Moore sequence stems from
Theorem \ref{theorem:MainIntroduction}. Ingredients from class field theory only enter in the
identification of the $K_{2}$-groups with groups of roots of unity. For the
middle term, this only uses tools from local class field theory, but showing
$K_{2}(\mathsf{LCA}_{F})_{/div}\cong\mu(F)$ relied on using global class field
theory, too.

\begin{remark}
[Wild kernel, \cite{MR2072396}]The kernel of the left-most horizontal maps in the above diagram is
understood. Historically, this kernel is known as the \emph{wild kernel}
$WK_{2}(F)$, as it is the joint kernel under all Hilbert symbols. The
nomenclature stems from the fact that the joint kernel of all tame symbols,
which agrees with $K_{2}(\mathcal{O}_{F})$, is known as the \emph{tame
kernel}. It is clear that the subgroup of divisible elements
$K_{2}(F)_{div}\subseteq K_{2}(F)$ must be contained in $WK_{2}(F)$. By Hutchinson \cite[Corollary 4.5]{MR1824144}, following ideas of Tate, the quotient $WK_{2}(F)/K_{2}(F)_{div}$ is either $0$ or
$\mathbb{Z}/2$, depending on the number field $F$. Both cases occur and it is
understood how one can distinguish these cases.
\end{remark}

\begin{remark}
\label{rmk_AtComplexPlaces}We complement the definition given for $m_{v}$ at
the beginning of \S \ref{subsect_ClassHRL} with a different viewpoint. For
each finite or real place $\mu(F_{v})$ is finite, so the profinite completion
$\mu(F_{v})\rightarrow\mu(F_{v})^{\wedge}$ is an isomorphism
and therefore $m_{v}=\#\mu(F_{v})$. Only for a complex place $\mu(F_{v})\cong\mathbb{Q}/\mathbb{Z}$ is infinite, and we have $(\mathbb{Q}%
/\mathbb{Z})^{\wedge}=0$. Hence, instead of using the profinite completion in
our definition of $m_{v}$, we could also agree on the convention that
$m_{v}\coloneqq 0$ for complex places, and $m_{v}=\#\mu(F_{v})$ for all noncomplex places.
\end{remark}

\section{Hilbert Reciprocity Law (non-commutative version)}

The considerations of the previous section suggest a kind of non-commutative
Moore sequence, and thus a non-commutative Hilbert Reciprocity Law. As we had
already seen in Construction \ref{construction_AbstractHReciprocityLaw}, the top row
of%
\[
\xymatrix{
K_2(F) \ar@{=}[d] \ar[r] & K_2(\mathsf{LCA}_{F,ab}) \ar[r]
\ar[d] & K_2(\mathsf{LCA}_{F}) \ar[r] \ar[d] & 0 \\
K_2(F) \ar[r] & \bigoplus_{v\enskip\operatorname{noncomplex}}{\mu(F_v)}
\ar[r] & \mu(F) \ar[r] & 0 \\
}
\]
is also available in the non-commutative setting. Let us explore this. Let $A$ be a
simple finite-dimensional $\mathbb{Q}$-algebra, let $F$ be its center. Then for every place $v$, $A_v$ is a central simple $F_v$-algebra.

We could hope that there is a nice concept of non-commutative
Hilbert symbols%
\[
h_{v}\colon K_{2}(A_{v})\longrightarrow(\ldots)
\]
with values in a yet unknown group $(\ldots)$ so that the $h_{v}$ satisfy an
interesting reciprocity law which would reduce to the ordinary Hilbert
Reciprocity Law if $A=F$. Indeed, it is not so difficult to identify what this
might be. As a first step, imitating the commutative situation, we could expect that the kernel of $h_{v}$ should be the
subgroup of divisible elements in $K_{2}(A_{v})$. \textit{Unless }$v$\textit{
is a real place such that }$A_{v}$\textit{ does not split over }$F_{v}$, the
reduced norm%
\[
K_{2}(A_{v})\overset{\operatorname*{nr}}{\longrightarrow}K_{2}(F_{v})
\]
is an isomorphism (in general, it is not: if $A$ are the rational Hamilton
quaternions, $A\otimes_{\mathbb{Q}}\mathbb{R}\cong\mathbb{H}$ so that
$K_{2}(\mathbb{H})=0$, yet $K_{2}(\mathbb{R})\cong\{\pm1\}$, so $-1$ cannot
possibly lie in the image of the reduced norm).

However, in all situations where the reduced norm is an isomorphism, the
subgroup of divisible elements in $K_{2}(A_{v})$ gets identified with the
divisible elements in $K_{2}(F_{v})$. Thus,%
\[
K_{2}(A_{v})/K_{2}(A_{v})_{div}\cong\mu(F_{v})\text{,}%
\]
where the isomorphism is given by the Hilbert symbol of the local field
$F_{v}$. Thus, the non-commutative Hilbert symbol reasonably would be taken to
be%
\begin{equation}
h_{v}^{\operatorname*{nc}}\colon K_{2}(A_{v})\longrightarrow\mu(F_{v}) \label{lzzx9}%
\end{equation}%
\[
h_{v}^{\operatorname*{nc}}\coloneqq h_{v}\circ\operatorname*{nr}\nolimits_{A_{v}}%
\]
and is merely the reduced norm, followed by the Hilbert symbol of the center.
Unlike the situation for number fields, this map $h_{v}^{\operatorname*{nc}}$
can fail to be surjective also for noncomplex places, but only in the rather
exceptional case where $v$ is a real place over which $A$ does not split.

One immediately deduces a non-commutative Hilbert Reciprocity Law with ease,
given the structure of the above definition.

\begin{proposition}
Let $A$ be a finite-dimensional simple $\mathbb{Q}$-algebra with center $F$.
Then for any $\alpha\in K_{2}(A)$ we have%
\[
\prod_{v\in\mathcal{P}\ell_{F}}h_{v}^{\operatorname*{nc}}(\alpha)^{\frac
{m_{v}}{m}}=1\text{,}%
\]
where $h_{v}^{\operatorname*{nc}}$ denotes the non-commutative\ Hilbert symbol
of Equation \eqref{lzzx9}. If $A=F$ the statement specializes to the classical
Hilbert Reciprocity Law of Equation \eqref{lzzp0}.
\end{proposition}

\begin{proof}
The proof is tautological by construction of $h_{v}^{\operatorname*{nc}}$. We
compute%
\[
\prod_{v\in\mathcal{P}\ell_{F}}h_{v}^{\operatorname*{nc}}(\alpha)^{\frac
{m_{v}}{m}}=\prod_{v\in\mathcal{P}\ell_{F}}h_{v}(\operatorname*{nr}%
\nolimits_{A_{v}}(\alpha\otimes A_{v}))^{\frac{m_{v}}{m}}\text{,}%
\]
where $\operatorname*{nr}_{A_{v}}\colon K_{2} (A_{v})\rightarrow K_{2}(F_{v})$ is the reduced norm and $\alpha\otimes F_{v}$ refers to the image of $\alpha\in K_{2}(A)$ in $K_{2}(A_{v})$. By Remark \ref{rmk_ReducedNormInjectivities} we have $\operatorname*{nr}_{A} (\alpha)\otimes F_{v}=\operatorname*{nr}_{A_{v}}(\alpha\otimes A_{v})$, so the
right term equals $h_{v}(\operatorname*{nr}_{A}(\alpha))^{\frac{m_{v}}{m}}$
and thus our claim reduces to the classical Hilbert Reciprocity Law, applied
to the element $\operatorname*{nr}_{A}(\alpha)\in K_{2}(F)$. If $A=F$, the
reduced norm is the identity map.
\end{proof}

Clearly, the above proposition is not very interesting because it reduces
everything to the commutative case. But we can push the analogy further. We would expect that Moore's
sequence holds analogously in the non-commutative setting, and should also
reduce to the commutative Moore sequence, except perhaps with some
modifications when the reduced norm fails to be surjective. This suggests the following.

\begin{conjecture}
\label{conj7}Suppose $A$ is a finite-dimensional simple $\mathbb{Q}$-algebra
with center $F$. Then%
\[
K_{2}(\mathsf{LCA}_{A})\cong K_{2}(\mathsf{LCA}_{F})\text{.}%
\]

\end{conjecture}

Curiously, when trying to prove this, we were naturally led to an old
conjecture of Merkurjev and Suslin, which is still open.

\begin{conjecture}
[Merkurjev--Suslin, 1982]\label{Conj_MerkurjevSuslin}Suppose $A$ is a
finite-dimensional central simple $F$-algebra, where $F$ is a local or global
field (say of characteristic zero). Then there is an exact sequence%
\begin{equation}
0\longrightarrow K_{2}(A)\overset{\operatorname*{nr}}{\longrightarrow}%
K_{2}(F)\longrightarrow\bigoplus_{v}\mathbb{Z}/2\mathbb{Z}\longrightarrow
0\text{,} \label{l_WangFullSequenceK2}%
\end{equation}
where $v$ runs over all real places of $F$ for which the algebra $A_{v}$ is
non-split. (If $F$ is a $p$-adic or complex local field, then the set of such
$v$ is necessarily empty).
\end{conjecture}

This conjecture is stated in \cite[Remark 17.5]{MR675529}. The conjecture is
fairly parallel to what is proven to happen for $K_{1}$ in view of the
Hasse--Schilling--Maass theorem \cite[Theorem 45.3]{MR892316}.

\begin{remark}
We should comment on the status of this conjecture:

\begin{enumerate}
\item (Merkurjev--Suslin) If $A$ has square-free index, the conjecture is
true. This is \cite[Theorem 17.4]{MR675529}.

\item (Merkurjev--Suslin, Kahn--Levine) If $A$ has square-free index, the
injectivity of $K_{2}(A)\overset{\operatorname*{nr}}{\longrightarrow}K_{2}(F)$
is actually true for all fields of characteristic zero. This was proven
independently by Merkurjev--Suslin \cite[Theorem 2.3]{MR2645334} and
Kahn--Levine \cite[Corollary 2]{MR2650808}. The latter exhibit an exact
sequence%
\[
0\longrightarrow K_{2}(A)\overset{\operatorname*{nr}}{\longrightarrow}%
K_{2}(F)\longrightarrow H^{4}(F,\mathbb{Z}/d(3))\longrightarrow H^{2}(F\left(
X\right)  ,\mathbb{Z}/d(3))\text{,}%
\]
where $X$ is the Severi--Brauer variety of $A$ and $d$ the index. Thus, if $F$
is a field of cohomological dimension $\leq3$, the reduced norm on $K_{2}$ is
an isomorphism for square-free $d$.

\item Sometimes one finds a mistaken claim in the literature attributing the
injectivity of the reduced norm on $K_{2}$, i.e. $SK_{2}(A)=0$, over number
fields to the paper \cite{MR743941}. However, this stems from a confusion and
this paper does not even claim to have shown this.
\end{enumerate}
\end{remark}

Assuming the Merkurjev--Suslin conjecture, we obtain a nice analogue of
Moore's sequence in the non-commutative setting which for $A=F$ reduces to the
original Moore sequence.  In here, we write $T^{\operatorname*{nc}}$ for the map as in \Cref{construction:ConstructionofT}, replacing the Hilbert symbol $h_v$ by $h_{v}^{\operatorname*{nc}}.$  Specifically, $T^{\operatorname*{nc}}$ is the composition
\[K_2(\mathsf{LCA}_{A,ab}) \xrightarrow{\cong} \prod\nolimits_v' K_{2}(A_{v}) \xrightarrow{h_{v}^{\operatorname*{nc}}} \bigoplus_{v\enskip\operatorname{noncomplex}}\mu(F_v)\text{.}\]

\begin{theorem}
Let $A$ be a simple finite-dimensional $\mathbb{Q}$-algebra. Write $F$ for its
center. Assume Sequence \ref{l_WangFullSequenceK2} is exact for $A$ and all
completions $A_{v}$ at the places of $F$ (e.g., if $A$ has square-free index).

\begin{enumerate}
\item The diagram%
\begin{equation}
\xymatrix{ K_2(A) \ar[d] \ar[r] & K_2(\mathsf{LCA}_{A,ab}) \ar[r] \ar[d]^{T^{\operatorname*{nc}}} & K_2(\mathsf{LCA}_{A}) \ar[r] \ar[d]^{\overline{T}^{\operatorname*{nc}}} & 0 \\ K_2(F) \ar[r] & \bigoplus
_{v\enskip\operatorname{noncomplex}}{\mu(F_v)}
\ar[r] & \mu(F) \ar[r] & 0, \\ }
\label{lsilo1}%
\end{equation}
commutes, where the bottom row is the Moore sequence (\Cref{theorem:MooreSequence}) and the top row comes from Theorem \ref{theorem:MainIntroduction}.  The arrow $T^{\operatorname*{nc}}$ arrow arises from the reduced norm, followed by the Hilbert symbol (defined above), and arrow $\overline{T}^{\operatorname*{nc}}$ is the induced map on the quotient.

\item The map $T^{\operatorname*{nc}}$ is an epimorphism after inverting $2$.  The kernel $\ker T^{\operatorname*{nc}}$ agrees with the subgroup of divisible elements in $K_{2}(\mathsf{LCA}_{A,ab})$.

\item Both rows in Diagram \ref{lsilo1} are exact.

\item Taking the profinite completion of the entire diagram, the left two
downward arrows become isomorphisms after inverting $2$, the right downward
arrow becomes an isomorphism on the nose.

\item $K_{2}(\mathsf{LCA}_{A})^{\wedge}\cong\mu(F)$.
\end{enumerate}
\end{theorem}

We repeat that if Conjecture \ref{Conj_MerkurjevSuslin} is true, the exactness
of Sequence \ref{l_WangFullSequenceK2} holds unconditionally. In this case,
the proof will also show the validity of Conjecture \ref{conj7}.

\begin{proof}
We first use Construction \ref{construction_AbstractHReciprocityLaw} twice: first
for $A$ itself and then for its center $F$. This yields the exact rows in the
following diagram%
\[
\xymatrix{
0 \ar[d] & 0 \ar[d] \\
K_2(A) \ar[d]_{\operatorname{nr} } \ar[r] & \prod^{\prime}K_{2}(A_{v}%
) \ar[r] \ar[d]_{\operatorname{nr} } & K_2(\mathsf{LCA}_{A})
\ar[r] \ar[d] & 0 \\
K_2(F) \ar[r] \ar[d] & \prod^{\prime}K_{2}(F_{v}) \ar[r] \ar[d] & K_2(\mathsf
{LCA}_{F}) \ar[r] & 0 \\
\bigoplus_{w}\mathbb{Z}/2 \ar[r] & \bigoplus_{w}\mathbb{Z}/2
}
\]
and we map the top row to the bottom row using the reduced norm on $K_{2}$. By
the commutative square in Remark \ref{rmk_ReducedNormInjectivities}, the top
left square commutes and the right square commutes by the exactness of the
rows. Next, as already indicated in the above diagram, we use Sequence
\ref{l_WangFullSequenceK2} to obtain the exact two columns on the left. As
both cokernels match, we deduce that the reduced norm induces an isomorphism%
\[
K_{2}(\mathsf{LCA}_{A})\underset{\operatorname*{nr}}{\overset{\cong
}{\longrightarrow}}K_{2}(\mathsf{LCA}_{F})\text{,}%
\]
i.e. we obtain Conjecture \ref{conj7}. All remaining claims now follow from
Theorem \ref{thm_numfieldsit}.
\end{proof}

If one could show that $K_{2}(\mathsf{LCA}_{A})\cong K_{2}(\mathsf{LCA}_{F})$
by a different method, this might be a promising start to reverse the logic in
order to prove Conjecture \ref{Conj_MerkurjevSuslin}. However, one might also
need some knowledge of the boundary map $K_{3}(\mathsf{LCA}_{A})\rightarrow
K_{2}(A)$.

\bibliographystyle{amsalpha}
\providecommand{\bysame}{\leavevmode\hbox to3em{\hrulefill}\thinspace}
\providecommand{\MR}{\relax\ifhmode\unskip\space\fi MR }
\providecommand{\MRhref}[2]{%
  \href{http://www.ams.org/mathscinet-getitem?mr=#1}{#2}
}
\providecommand{\href}[2]{#2}

\end{document}